\documentclass[10pt]{article}
\usepackage[T1]{fontenc}
\usepackage[utf8]{inputenc}
\usepackage[tt=false]{libertine}

\usepackage{amsmath,amssymb,amsthm,mathtools}

\usepackage[varbb]{newpxmath}

\usepackage{bm}

\usepackage[varbb]{newpxmath}

\usepackage{booktabs}
\usepackage{enumitem}
\usepackage{natbib}
\usepackage[hidelinks]{hyperref}
\usepackage{microtype}
\usepackage{xcolor}
\usepackage{multirow}

\mathtoolsset{showonlyrefs=true}
\numberwithin{equation}{section}

\theoremstyle{plain}

\newtheorem{theorem}{Theorem}[section]
\newtheorem{lemma}[theorem]{Lemma}
\newtheorem{proposition}[theorem]{Proposition}
\newtheorem{corollary}[theorem]{Corollary}
\theoremstyle{definition}

\newtheorem{assumption}[theorem]{Assumption}
\newtheorem{remark}[theorem]{Remark}

\newcommand{\R}{\mathbb{R}}

\newcommand{\E}{\mathbb{E}}
\newcommand{\Prob}{\mathbb{P}}
\newcommand{\dd}{\mathrm{d}}
\newcommand{\KL}{\mathrm{KL}}
\newcommand{\TV}{\mathrm{TV}}

\newcommand{\NN}{\mathrm{NN}}

\newcommand{\1}{\mathbf{1}}
\newcommand{\norm}[1]{\left\lVert #1\right\rVert}
\newcommand{\abs}[1]{\left\lvert #1\right\rvert}

\newcommand{\cX}{\mathcal{X}}

\newcommand{\cK}{\mathcal{K}}
\newcommand{\cA}{\mathcal{A}}

\newcommand{\cH}{\mathcal{H}}

\newcommand{\cP}{\mathcal{P}}

\newcommand{\var}{\mathrm{Var}}
\newcommand{\ve}{\varepsilon}
\newcommand{\wh}{\widehat}

\newcommand{\V}{V}
\newcommand{\requ}{\sigma_{\mathrm{ReQU}}}
\usepackage{authblk}

\title{Nonparametric inference for density-dependent McKean--Vlasov diffusions}
\author[1]{Denis Belomestny\thanks{Email: denis.belomestny@uni-due.de}}
\author[1]{Ekaterina Morozova\thanks{Corresponding author. Email: ekaterina.morozova@uni-due.de}}

\affil[1]{Duisburg-Essen University, 
Essen, Germany}

\date{}

\begin{document}
\maketitle

\begin{abstract}
The present research is devoted to the nonparametric estimation of a density-dependent drift coefficient in a multivariate McKean--Vlasov diffusion from independent observations at a common time, as well as the stationary density.
Under certain assumptions on the (known) potential, we reduce the problem to the one-dimensional one and construct a sieve maximum-likelihood estimator based on sparse ReQU neural networks subject to structural and H\"older constraints. 
Using the endpoint-adapted graded approximation, we achieve the rate of 
\(
\left(b_n\log n/n\right)^{2(\beta+1)/(2\beta+3)}
\)
for the Kullback-Leibler divergence between the true and estimated stationary densities, with \(b_n\) being at most a logarithmic factor. Similarly, it is shown that the constructed estimator for the drift coefficient converges to the true one at the rate of 
\(
\left(b_n\log n/n\right)^{\beta/(2\beta+3)}
\)
in the \(L^2\)-metric. A matching Assouad lower bound proves minimax optimality of this bound up to logarithmic factors.

\medskip
\noindent\textbf{Keywords:} McKean--Vlasov SDEs, non-parametric estimation, deep neural networks, invariant measure, maximum likelihood. \\
\noindent\textbf{MSC 2020:} 62G05, 62G20, 60H10, 60J60.
\end{abstract}

\section{Introduction}

We consider the nonparametric recovery of a density-dependent drift coefficient from independent particles observed once at a common time. The dynamics are \(d\)-dimensional and nonlinear in the unknown density, yet the available data are only cross-sectional. The central question is whether --- and at what rates --- the density-dependent function \(\Xi\) and the induced stationary density \(\pi_\Xi\) can be reconstructed from this stationary or near-stationary experiment.

Let \(p_t\) denote the Lebesgue density of \(X_t\). We study the McKean–Vlasov stochastic differential equation
\begin{equation}\label{eq:model-general}
  \dd X_t=-\Xi(p_t(X_t))\nabla \V(X_t)\,\dd t+\sqrt{2}\,\dd W_t,
  \qquad X_0\sim p_0,
\end{equation}
where \(W\) is a standard Brownian motion in \(\R^d\), \(\V:\R^d\to\R\) is a known confining potential, and \(\Xi:\R_+\to\R_+\) is the interaction function which is unknown. The corresponding nonlinear Fokker--Planck equation is 
\(
\partial_t p_t-\Delta p_t- \nabla\!\cdot\!\left(\nabla \V(x)\Xi(p_t(x))p_t(x)\right)=0.
\)
Unlike the usual McKean--Vlasov models involving moments, kernels, or cumulative distributions, the drift in~\eqref{eq:model-general} depends directly on the pointwise density. This dependence is not continuous under weak perturbations of the law and leads to a singular statistical problem.

Since, as shown in Section~\ref{sec:model}, 
the zero-flux stationary density satisfies 
\(\pi_\xi(x)=q_\xi(V(x))\)
with \(q_\xi\) solving~\eqref{eq:q-ode-general}, 
the likelihood depends on an observation 
\(X\) only through the scalar energy \(V(X)\). Estimating \(\Xi\) therefore becomes a one-dimensional nonlinear inverse problem, which, however, is ill-posed in the low-density regime corresponding to the spatial tails, and the coefficient is identifiable only over the range attained by the stationary density. These features determine the achievable estimation rates and require approximation methods adapted to the tail regime.

Most existing statistical results for McKean-Vlasov models concern smoother functionals of the law or observations of trajectories and interacting particle systems. The closest likelihood-based approach is that of~\cite{belomestny-orlova-2025}, where the interaction depends on the cumulative distribution function. The pointwise density dependence in~\eqref{eq:model-general} instead requires endpoint-adapted approximation and a separate inverse-stability analysis. The analytic properties of the model for the Ornstein-Uhlenbeck potential 
\begin{equation}\label{eq:ou-potential}
  \V(x)=1+\norm{x}_2^2, \qquad \nabla \V(x)=2x,
\end{equation}
were studied by~\cite{belomestny-morozova-ou}; here we develop the corresponding nonparametric statistical theory for a broader class of confining potentials. Our main contributions are as follows.

\begin{enumerate}
\item
We establish the representation~\eqref{eq:pi-general} 
for zero-flux stationary densities and reduce the problem from \(d\)-dimensional to a one-dimensional one.
\item
We construct a constrained ReQU sieve maximum-likelihood estimator and develop an endpoint-adapted approximation yielding
\[
\KL(\pi_\xi\Vert\pi_{\widehat\Xi_n})
=O_{\mathbb{P}}
\left(
\left(\frac{b_n\log n}{n}
\right)^{
\frac{2(\beta+1)}{2\beta+3}
}
\right),
\quad 
\|\widehat\Xi_n-\Xi\|_{L^2(I)}
=
O_{\mathbb P}
\left(
\left(\frac{b_n\log n}{n}\right)^{
\frac{\beta}{2\beta+3}
}
\right)
\]
for every \(I\Subset(0,\max_{x\in\R^d}\pi_\xi(x))\),
where \(b_n=O(1)\) in the compact (i.e., bounded domain) setting and \(b_n\lesssim\log n\) in the noncompact one.
\item 
We prove a matching Assouad lower bound of order
\(
n^{-2\beta/(2\beta+3)}
\)
for the squared \(L^2(I)\)-risk. The construction is carried out in the primitive coordinate \(\rho_\xi(r)=(r\xi(r))^{-1}\) and shows that the coefficient rate is minimax optimal up to logarithmic factors.
\item
Finally, we transfer the stationary guarantees to finite-time observations under quantitative convergence to equilibrium; for the Ornstein-Uhlenbeck model, \(T_n\asymp\log n\) is sufficient.
\end{enumerate}

The remainder of the paper is organized as follows. Section~\ref{sec:literature} discusses related work. Section~\ref{sec:model} derives the stationary representation and likelihood reduction. Section~\ref{sec:dnn-parametrization} constructs the neural-network sieve and graded approximation. Section~\ref{sec:klbernstein} proves the likelihood oracle inequality, while Section~\ref{sec:upper-exact} derives the estimation rates and their finite-time extension. Section~\ref{sec:lower} establishes the minimax lower bounds, and Section~\ref{sec:sim-study} presents the numerical study. The proofs and auxiliary results are collected in the Appendix.

\section{Literature Review}\label{sec:literature}

MVSDEs were introduced by McKean~\citeyearpar{mckean1966} and are the standard framework for mean-field limits; see, e.g., \cite{sznitman1991}. Statistical inference for these systems includes maximum likelihood for interacting systems (\cite{kasonga1990}), semiparametric and parametric methods (\cite{belomestny-pilipauskaite-podolskij-2023,genon-catalot-laredo-2024}), online estimation (\cite{sharrock-kantas-parpas-pavliotis-2023}), and nonparametric methods (\cite{della-maestra-hoffmann-2022,comte-genon-catalot-2023}). The conservation-law model studied by ~\cite{jourdain-malrieu-2008} and~\cite{belomestny-orlova-2025} is driven by the cumulative distribution function of the law. The latter introduced an ODE-induced likelihood framework over noncompact invariant densities. Model \eqref{eq:model-general} is of Nemytskii type: the drift depends directly on $p_t(x)$. These mappings are not continuous under weak perturbations, so PDE techniques are required. Related work includes nonlinear Fokker--Planck flows (\cite{barbu-rockner-2024}), time-dependent Nemytskii MVSDEs (\cite{grube2024}), weighted $L^1$ semigroups (\cite{rehmeier2023}), and equations with unbounded coefficients (\cite{bogachev-salakhov-shaposhnikov-2024}). We rely on the analysis of the OU potential from~\cite{belomestny-morozova-ou}. Our deep neural network approximation uses the results of~\cite{belomestny-naumov-puchkin-samsonov-2023} on piecewise-polynomial activations. ReQU networks are useful here because differentiability is required for both the model and the inverse problem. As in~\cite{belomestny-orlova-2025}, we intersect the network class with a deterministic smoothness ball because sparsity and bounded weights alone do not provide inverse stability. Our lower bounds use Assouad methods from~\cite{tsybakov2009}; related DNN nonparametric rates appear in~\cite{schmidt-hieber-2020}. 

\section{Model, stationary representation, and estimation idea}\label{sec:model}

\subsection{Model assumptions}

In what follows, it is assumed that the potential \(V\) is known and satisfies the following assumption. 
\begin{assumption}\label{ass:potential}
The function \(\V:\R^d\to[v_\star,\infty)\) is \(C^2\) and coercive, \(v_\star:=\min_{x\in\R^d} V(x)\), and \(e^{-a\V}\) is integrable for any \(a>0\). Its level-set measure 
\(m_\V(s)
:=\int_{\{x:\V(x)=s\}}\norm{\nabla \V(x)}_2^{-1}\mathcal H^{d-1}(\dd x)
\) 
for \(s>v_\star\) is locally finite and bounded above and below by positive constants on compact intervals away from critical values.
\end{assumption}
\begin{remark}
    It can be seen that the Ornstein-Uhlenbeck potential~\eqref{eq:ou-potential} satisfies Assumption~\ref{ass:potential} with 
    \(v_\star=1\) and
    \(
    m_V(s)
    =(\omega_d/2)(s-1)^{d/2-1},
    \)
    where
    \(
    \omega_d=2\pi^{d/2}/\Gamma(d/2)
    \),
    \(s>1\),
which makes it a prime example of models of the considered class.
\end{remark}

Furthermore, we work under the following assumption on the coefficient class.
\begin{assumption}\label{ass:coefficient}
There exist constants $0<\kappa<K<\infty$ and $L<\infty$ such that every admissible coefficient \(\xi\), including the true coefficient \(\Xi_0\), satisfies
\begin{equation}\label{eq:bounds-xi}
 \kappa\leq \xi(r)\leq K \quad \text{and} \quad \abs{\xi(r)-\xi(u)}\leq L\abs{r-u} \quad \text{for all } r,u\geq0.
\end{equation}
\end{assumption}
The structural conditions above are sufficient for the stationary representation and its basic stability properties. For the approximation and rate results, we impose additional regularity only on the true coefficient over the density range relevant to estimation. More precisely, for some \(\beta>1\) and \(H<\infty\), we assume \(\Xi_0\in\mathcal H^\beta([0,U],H)\), where \(U\) is chosen in Subsection~\ref{subsec:requ} to contain the density ranges of all candidate stationary models. We also assume that the range of \(\Xi_0\) is separated from the boundary of the working envelope \([\kappa,K]\), i.e., there exists \(\eta_0>0\) such that
\begin{equation}\label{eq:truth-interior-envelope}
 \kappa+\eta_0\leq\Xi_0(r)\leq K-\eta_0 \quad \text{for } 0\leq r\leq U.
\end{equation}
Condition~\eqref{eq:truth-interior-envelope} is an interiority condition relative to the sieve envelope that will be constructed later, rather than an additional shape restriction on \(\Xi_0\). The constants \(\kappa\) and \(K\) need not be sharp and may be chosen slightly outside the range of the true coefficient.

\subsection{Invariant density and identifiability}

For a candidate $\xi$, define
\begin{equation}\label{eq:gxi}
 g_\xi(r)=\int_1^r (u\xi(u))^{-1}\,\dd u, \qquad r>0.
\end{equation}

\begin{proposition}
\label{prop:invariant-general}
Under Assumptions~\ref{ass:potential} and \ref{ass:coefficient}, there exists a unique \(C^1\) zero-flux stationary density for $\xi$ given by
\begin{equation}\label{eq:pi-general}
 \pi_\xi(x)=g_\xi^{-1}(\mu_\xi-\V(x)),
\end{equation}
where $\mu_\xi$ is the unique constant satisfying \( \int_{v_\star}^{\infty}g_\xi^{-1}(\mu_\xi-s)m_\V(s)\,\dd s=1 \). Equivalently, $q_\xi(s):=g_\xi^{-1}(\mu_\xi-s)$ solves
\begin{equation}\label{eq:q-ode-general}
 q_\xi'(s)=-q_\xi(s)\xi(q_\xi(s)) \quad \text{for} \quad s>v_\star,
\end{equation}
where $a_\xi:=q_\xi(v_\star)$ is chosen to satisfy 
\begin{equation}\label{qmv_norm}
\int_{v_\star}^{\infty}q_\xi(s)m_\V(s)\,\dd s=1.
\end{equation}
\end{proposition}
The proof is in Appendix~\ref{proof:prop:invariant-general}.
\begin{remark}
    For the OU potential~\eqref{eq:ou-potential},~\cite{belomestny-morozova-ou} prove this zero-flux solution is the unique stationary density.
\end{remark}

\begin{remark}
By~\eqref{eq:q-ode-general}, integrating 
\((\log q_\xi)'=-\xi(q_\xi)\) one obtains the Gaussian envelope
\begin{equation}\label{eq:ou-envelope}
 a_\xi e^{-K(s-v_\star)}\leq q_\xi(s)\leq a_\xi e^{-\kappa(s-v_\star)}, \qquad s\geq v_\star.
\end{equation}
Defining \(Z_c:=\int_{\R^d} \exp(-c(V(x)-v_\star))\,dx\) for \(c>0\) and using the above bounds alongside the normalisation condition~\eqref{qmv_norm}, we get that \(Z_\kappa^{-1}\leq a_\xi\leq Z_K^{-1}\), i.e.,
\begin{equation}\label{eq:exp-envelope_gen}
\frac{1}{Z_\kappa} e^{-K(s-v_\star)}
\leq q_\xi(s)
\leq \frac{1}{Z_K} e^{-\kappa(s-v_\star)},
\quad s\geq v_\star.
\end{equation}
In particular, for the Ornstein-Uhlenbeck potential~\eqref{eq:ou-potential} it holds that 
\(
(\kappa/\pi)^{d/2} 
\leq a_\xi 
\leq (K/\pi)^{d/2}.
\)
\end{remark}

It is worth mentioning that stationary observations identify \(\Xi_0\) only on the range \((0,a_0]\) of \(\pi_0\), where 
\(a_0=q_{\Xi_0}(v_\star)\) is the maximal stationary density. The limit \(r\to0\) corresponds to the spatial tail, whereas \(r\to a_0\) corresponds to the minimum energy level. We therefore fix an interior interval
\(
I=[r_-,r_+]\Subset(0,a_0)
\)
as the domain on which the coefficient is recovered.

\subsection{Dimension reduction and maximum likelihood estimation}
\label{subsec:mle}

Observe that, if \(X\sim\pi_\xi\), then \(S=V(X)\) has density 
\(f_\xi(s)=q_\xi(s)m_V(s)\). 
Given \(S=s\), the distribution of \(X\) on the level set \(\{\V=s\}\) does not depend on \(\xi\). Thus, the likelihood for \(\xi\) based on \(X_1,\ldots,X_n\) is equivalent (up to a constant) to the likelihood based on \(S_i=\V(X_i)\). 
Given a sample \(X_1,\dots,X_n\) from \(\pi_0\) and a sieve \(\mathfrak X_m\) 
satisfying \eqref{eq:bounds-xi}, we hence define the maximum likelihood estimator as
\begin{equation}\label{eq:mle}
\wh\Xi_n\in\arg\min_{\xi\in\cX_m}\mathcal L_n(\xi), 
 \quad \text{where} \quad \mathcal L_n(\xi)=\frac{1}{n}\sum_{i=1}^n -\log \pi_\xi(X_i),
\end{equation}
which is equivalent to minimizing 
\(
\frac{1}{n}\sum_{i=1}^n -\log q_\xi(S_i)
\), where \(S_i=\V(X_i)\). 
For each candidate \(\xi\), the likelihood is numerically evaluated by solving the scalar profile equation with an initial value chosen through the normalization constraint~\eqref{qmv_norm} and interpolating the resulting profile at the observed energy values \(S_i=V(X_i)\). 

\section{Neural-Network Sieve and Graded Approximation}\label{sec:dnn-parametrization}
\subsection{ReQU sieve and structural constraints}\label{subsec:requ}
Since the candidate modes $a_\xi$ can exceed the true mode $a_0$, the neural network domain must accommodate all empirically reachable densities. Given~\eqref{eq:exp-envelope_gen}, we fix the domain upper bound \(U\geq \max\{1, a_+\}\), where \(a_+ := Z_K^{-1}\).
The value $1$ is included as the exact anchor in the primitive \eqref{eq:gxi}. For $r > U$, the networks are structurally extended by the constant value $\xi_\theta(U)$, ensuring that the coefficient bounds in \eqref{eq:bounds-xi} hold globally.

Let \(\requ(z):=(z_+)^2\)
denote the ReQU activation, and let \(\NN(L,\bm p,s,B)\) be the class of ReQU networks with depth \(L\), width vector \(\bm p\), at most \(s\) nonzero parameters, and all parameters bounded in absolute value by \(B\). For the \(m\)-th sieve, the depth \(L\) is fixed, while the remaining architectural parameters may grow subject to \(\max_j p_{m,j}\leq C_p m\), \(s_m=\lceil C_s m\rceil\) and \(B_m\leq m^{C_B}\) for fixed constants \(C_p, C_s,C_B>0\) independent of \(m\). These restrictions provide the metric-entropy control required below.
To enforce the structural bounds \(\kappa\leq\xi\leq K\), fix some 
\(0<\tau<\eta_0/4(K-\kappa)\)
and let \(S_\tau:\R\to[0,1]\) be a globally Lipschitz clamping function having a Lipschitz constant \(L_\tau<\infty\) satisfying \(S_\tau(z)=z\) for all \(z\in[\tau,1-\tau]\). We parameterize candidate interaction coefficients by
\begin{equation}\label{eq:dnn-xi}
 \xi_\theta(r)=\kappa+(K-\kappa)S_\tau(N_\theta(r)) \quad \text{for } r\in[0,U], \quad N_\theta\in\NN(L,\bm p_m,s_m,B_m)
\end{equation}
and define the corresponding sieve
\begin{equation}\label{eq:smooth-dnn-sieve}
 \cX_m := \bigl\{\xi_\theta:\text{defined by \eqref{eq:dnn-xi}}\bigr\} \cap \mathcal H^\beta([0,U],H_1)
\end{equation}
with \(H_1>0\) a sufficiently large constant whose existence is established in Lemma~\ref{lem:dnn}. The estimator for \(\Xi_0\) is then given by 
\(\widehat{\Xi}=\xi_{\widehat{\theta}}\), where
\(
\widehat{\theta}
\in
\arg\min_{\theta:\,\xi_\theta\in\cX_m}
\mathcal L_n(\xi_\theta).
\)

\subsection{Linearized forward map and endpoint-adapted approximation}

Since the likelihood measures approximation error through the induced stationary densities rather than directly through the coefficient, the relevant approximation term for the analysis is \(\KL(\pi_0\Vert \pi_{\xi_m})\). 
Given~\eqref{eq:pi-general}, the change that is brought to the stationary model by coefficient approximation is characterised by the function  \(g_\xi\) and the corresponding normalisation constant \(\mu_\xi\). More precisely, define a small perturbation \(\xi_t:=\Xi_0+th\). Hereafter, a subscript zero indicates evaluation at the true coefficient \(\Xi_0\). Define
\(
(\cK h)(r)
:=
\int_1^r h(u)w_0(u)\,\dd u
\)
with \(w_0(u):=(u\Xi_0(u)^2)^{-1}\)
and observe that, since 
\(
\frac{\dd}{\dd t}|_{t=0}g_{\xi_t}(r)
=
-(\cK h)(r),
\)
differentiating
\(
g_{\xi_t}(q_{\xi_t}(s))=\mu_{\xi_t}-s
\)
gives
\begin{equation}\label{eq:stat-time}
\frac{\dd}{\dd t}\bigg|_{t=0}
\log\pi_{\xi_t}(x)
=
\Xi_0(\pi_0(x))
\left(
\dot\mu_0+(\cK h)(\pi_0(x))
\right),
\end{equation}
where 
\(\dot\mu_0=\partial_t\mu_{\xi_t}|_{t=0}\). Differentiating also the normalisation condition~\eqref{qmv_norm} gives
\begin{multline}\label{eq:c0-functional}
 \dot\mu_0
=-\left(\int_{v_\star}^{\infty}q_0(s)\Xi_0(q_0(s))(\cK h)(q_0(s))m_\V(s)\,\dd s\right)
\\\times\left(\int_{v_\star}^{\infty}q_0(s)\Xi_0(q_0(s))m_\V(s)\,\dd s\right)^{-1}.
\end{multline}
Hence, we arrive at the score operator
\(
 (\cA h)(r)
 =\Xi_0(r)\left(
 c_0(h)+(\cK h)(r)
 \right)
\)
with \(c_0(h):=\dot\mu_0\).
Now let \(X\sim\pi_0\) and define \(R:=\pi_0(X)\), denoting the law of \(R\) by \(\nu_0\). With
\(
s_0(r)=q_0^{-1}(r)=\mu_0-g_0(r),
\)
a change of variables gives
\begin{equation}\label{v0}
(\dd\nu_0/\dd r)(r)
=
m_V(s_0(r))/\Xi_0(r),
\qquad 0<r<a_0,
\end{equation}
from which the local Kullback--Leibler geometry is
\[
\KL(\pi_0\Vert\pi_{\xi_t})
=
\frac{t^2}{2}
\E_0\left[(\cA h)(R)^2\right]
+o(t^2)
=
\frac{t^2}{2}
\Vert\cA h\Vert_{L^2(\nu_0)}^2
+o(t^2).
\]
Consequently, for an approximant \(\xi_m\), the relevant forward error
is
\(
\Vert\cA(\xi_m-\Xi_0)\Vert_{L^2(\nu_0)}
\).

To determine the approximation accuracy required near the endpoints, we examine the measure appearing in the forward norm. By~\eqref{eq:c0-functional} and the uniform bounds on \(\Xi_0\), controlling
\(
\Vert\cA h\Vert_{L^2(\nu_0)}
\)
reduces, up to fixed constants, to controlling
\(
\Vert\cK h\Vert_{L^2(\nu_0)}.
\)
The behaviour of the density of \(\nu_0\) near \(r=0\) therefore determines the appropriate approximation mesh. In turn, the endpoint behaviour of \(\nu_0\) is determined by the growth of the coarea factor \(m_V(s)\) at large energies. Indeed, since
\[
\frac{d\nu_0}{dr}(r)
=\frac{m_V(s_0(r))}{\Xi_0(r)}
\quad\text{and}\quad
s_0(r)-v_\star
=\int_r^{a_0}\frac{du}{u\Xi_0(u)},
\]
it holds that
\[
\frac{1}{K}\log\frac{a_0}{r}
\leq s_0(r)-v_\star
\leq\frac{1}{\kappa}\log\frac{a_0}{r}.
\]
In what follows, we are making the following assumption on the potential \(V\).
\begin{assumption}\label{ass:log}
There exist \(C_V<\infty\), \(\alpha\geq0\), and \(v_V>v_\star\) such that
\[
m_V(s)\leq C_V(1+s-v_\star)^\alpha,
\qquad s\geq v_V.
\]
\end{assumption}
Under this condition, after decreasing \(r_0>0\) if necessary,
\begin{equation}\label{eq:nu-zero-upper}
v_0(r)
\leq
C_\nu
\left(
1+\log\!\left(\frac{r_0}{r}\right)
\right)^{\alpha},
\qquad
0<r\leq r_0,
\end{equation}
with some \(0<C_\nu<\infty\) and \(\alpha\geq 0\).
For instance, for the Ornstein--Uhlenbeck potential~\eqref{eq:ou-potential}, since
\(m_V(s)=0.5\omega_d(s-1)^{d/2-1}\), condition~\eqref{eq:nu-zero-upper} holds
with \(\alpha=(d/2-1)_+\) and \(
r_0<\min\{a_0,1/2\}
\).
Hence, under Assumption~\ref{ass:log}, \(\nu_0\) contributes at most a logarithmic factor near zero; the principal loss in the forward approximation rate comes instead from the factor \(u^{-1}\) in \(\cK\). If \(h_m=\xi_m-\Xi_0\) satisfies
\(
|h_m(u)|\lesssim u^\beta
\)
on the first mesh cell \([0,\Delta]\), then
\[
\Vert\cK h_m\Vert_{L^2(\nu_0;[0,\Delta])}
\lesssim
\Delta^{\beta+1/2}
\left(1+\left|\log\Delta\right|\right)^{\alpha/2}.
\]
A uniform mesh, for which \(\Delta\asymp m^{-1}\), therefore yields only the order \(m^{-(\beta+1/2)}\), up to logarithmic factors.

To recover the desired forward approximation order, we use the graded preliminary knots
\(
\widetilde r_{j,m}=U(j/m)^\vartheta
\)
and align the nearest knot with \(1\). This gives
\begin{equation}\label{eq:graded-knots}
0=r_{0,m}<r_{1,m}<\cdots<r_{m,m}=U,
\qquad
1\in\{r_{j,m}:0\leq j\leq m\}.
\end{equation}
The first cell then has length
\(
\Delta_{1,m}\asymp m^{-\vartheta}.
\)
Choosing
\(
\vartheta>
(2\beta+2)/(2\beta+1)
\)
ensures that its contribution is of order at most \(m^{-(\beta+1)}\), up to logarithmic factors. The following lemma constructs a ReQU approximant on this mesh and simultaneously controls its uniform coefficient error, its forward error in \(L^2(\nu_0)\), and the complexity of the resulting sieve.

\begin{lemma}
\label{lem:dnn}
Assume \(\Xi_0\in\mathcal H^\beta([0,U],H)\) satisfies \eqref{eq:truth-interior-envelope}, and Assumption~\ref{ass:log} holds. 
Consider the graded mesh with 
\(
\vartheta>
(2\beta+2)/(2\beta+1)
\).
There exist fixed architecture constants 
\(L, C_p, C_s, C_B\)
and 
\(m_0, H_1<\infty\) 
such that, for every \(m\geq m_0\), there exists a candidate \(\xi_m \in \cX_m\) 
satisfying the exact moment cancellations on every cell, namely,
\begin{equation}\label{eq:approximant-moments}
 \int_{J_{j,m}}\left( 
 \xi_m(u)-\Xi_0(u)
 \right) 
 \frac{\dd u}{u\Xi_0(u)^2}=0.
\end{equation}
This candidate achieves
\begin{equation}\label{eq:app_fwd}
\|\xi_m-\Xi_0\|_{\infty,[0,U]}
\leq C_{\mathrm{app}}m^{-\beta}
\quad\text{and}\quad
\|\mathcal A(\xi_m-\Xi_0)\|_{L^2(\nu_0)}
\leq C_{\mathrm{fwd}}m^{-(\beta+1)}
\end{equation}
with
\(
C_{\mathrm{app}}
=C_{\mathrm{app}}(\beta,H,U,\kappa,K,\vartheta),
\)
and
\(
C_{\mathrm{fwd}}
=C_{\mathrm{fwd}}
(\beta,H,U,\kappa,K,\vartheta,r_0,\alpha,C_\nu)
\).
Moreover, for \(\varepsilon\in(0,1]\), 
the sieve entropy satisfies 
\[
\log N(\varepsilon,\mathfrak X_m,\|\cdot\|_\infty)
\leq
C_{\mathrm{ent}}m
\log\left(C_{\mathrm{ent}}m/\varepsilon\right),
\]
where \(C_{\mathrm{ent}}\) depends only on the fixed sieve-architecture constants, 
\(U\), \(\kappa\), \(K\) and
\(L_\tau\).
\end{lemma}
The proof is given in Appendix~\ref{proof:lem:dnn}.

\section{Fast Oracle Inequality for the Stationary Density}
\label{sec:klbernstein}

Now we aim at bounding the KL-error of the MLE~\eqref{eq:mle}. The fast density rate follows the localization strategy of
\cite{belomestny-orlova-2025}. For \(M>v_\star\), let
\[
A_M:=\{x:V(x)\leq M\},
\quad
\tau_\xi(M):=\int_{A_M^c}\pi_\xi(x)\,\dd x,
\qquad
\pi_\xi^M(x):=
\frac{\pi_\xi(x)\mathbf 1_{A_M}(x)}
{1-\tau_\xi(M)}.
\]
The general exponential envelopes~\eqref{eq:exp-envelope_gen} imply that \(\tau_\xi(M)\) decreases
exponentially in \(M\), uniformly over the coefficient class, while the
log-likelihood ratios on \(A_M\) grow at most linearly in \(M\). Thus,
choosing \(M_n-v_\star\asymp\log n\) makes the truncation terms
polynomially small at the cost of logarithmic factors. We first analyse the likelihood for the truncated densities. Let
\(Y_1,\ldots,Y_n\) be i.i.d.\ from a density \(p_0\), and write
\[
P_n f:=\frac1n\sum_{i=1}^n f(Y_i),
\qquad
P_0f:=\int_{\R^d} f(y)p_0(y)\,\dd y.
\]
If \(p^\ast\) is a comparison density in the truncated sieve and
\(\widehat p\) satisfies
\(
P_n\log\widehat p
\geq
P_n\log p^\ast-\varepsilon_{\mathrm{opt},M},
\)
where \(\varepsilon_{\mathrm{opt},M}\geq0\) is the optimization tolerance measuring how far \(\widehat p\) is from maximizing the truncated empirical log-likelihood,
then
\[
\operatorname{KL}(p_0\|\widehat p)
-
\operatorname{KL}(p_0\|p^\ast)
\leq
(P_n-P_0)\log\frac{\widehat p}{p^\ast}
+
\varepsilon_{\mathrm{opt},M}.
\]
Moreover,
\[
(P_n-P_0)\log\frac{\widehat p}{p^\ast}
=
(P_n-P_0)\log\frac{p_0}{p^\ast}
-
(P_n-P_0)\log\frac{p_0}{\widehat p}.
\]
The problem is therefore reduced to controlling centered log-likelihood
ratios uniformly over the truncated sieve. After discretizing the sieve
by a finite net, Lemma~\ref{lem:kl-bernstein} bounds each such fluctuation by an absorbable
multiple of the corresponding KL divergence plus an entropy term of
order \(B_{\gamma_M}/n\), yielding the fast truncated oracle inequality.
The subsequent full--truncated decomposition then adds the normalization,
tail, and escape-probability terms required to return to the full MLE.

\subsection{Localized likelihood control on the truncated domain}

We start with the following corrected localized KL--Bernstein inequality.
\begin{lemma}
\label{lem:kl-bernstein}
Let \(Y_1,\ldots,Y_n\) be i.i.d.\ random variables from a density \(p_0^A\) on a set \(A\). Let \(\cP_A\) be a finite set of densities such that \(\gamma\leq p_0^A(y)/p^A(y)\leq\gamma^{-1}\) for some \(\gamma\in(0,1)\). Let  \(\Lambda_\gamma=\log(1/\gamma)\) and 
\begin{equation}\label{eq:B-gamma}
 B_\gamma:=\frac{\Lambda_\gamma^2}{\Lambda_\gamma-1+\gamma}.
\end{equation}
With probability at least \(1-\delta\), simultaneously for all \(p^A\in\cP_A\) it holds that
\begin{multline}
 \left| \frac1n\sum_{i=1}^n\log\frac{p_0^A(Y_i)}{p^A(Y_i)} -\E_0\log\frac{p_0^A(Y_1)}{p^A(Y_1)} \right| 
 \leq
 \sqrt{\frac{2B_\gamma\KL(p_0^A\Vert p^A)\left(\log|\cP|+\log(2/\delta)\right)}{n}} \\
 \quad +\frac{2\Lambda_\gamma\left(\log|\cP|+\log(2/\delta)\right)}{3n}.
 \label{eq:kl-bernstein}
\end{multline}
Consequently, for any \(\eta\in(0,1)\),
\begin{equation}\label{eq:kl-bernstein-linearized}
 \left|(P_n-P_0)\log\frac{p_0^A}{p^A}\right| \leq \eta\KL(p_0^A\Vert p^A) + C_\eta \frac{B_\gamma\left(\log|\cP|+\log(2/\delta)\right)}{n}.
\end{equation}
\end{lemma}
The proof is deferred to Appendix~\ref{proof:lem:kl-bernstein}.

The first term on the right-hand side of~\eqref{eq:kl-bernstein-linearized} can be absorbed into the likelihood excess risk. This localization is the source of the fast \(1/n\), rather than \(1/\sqrt n\), stochastic order. We also note that the constant in Lemma 6.1 of~\cite{belomestny-orlova-2025} should be replaced by \(B_\gamma\); this correction does not affect the resulting rate.

To apply Lemma~\ref{lem:kl-bernstein} to the infinite sieve, let
\(
\mathcal P_m^M=\{\pi_\xi^M:\xi\in\cX_m\}
\)
and equip this class with the log-density distance
\(
\mathfrak d_M(p,q):=\norm{\log p-\log q}_{\infty,A_M}.
\)
Let \(\mathcal N_{m,\varepsilon}^M\) be an \(\varepsilon\)-net of \(\mathcal P_m^M\) under \(\mathfrak d_M\), and write
\(\mathcal H_m(\ve,M):= \log |\mathcal N_{m,\varepsilon}^M|\).
The error incurred when replacing a candidate by its nearest net element is measured by
\begin{equation}\label{eq:omega-modulus}
\omega_m(\varepsilon,M)
=
\sup_{p\in\mathcal P_m^M}
\inf_{q\in\mathcal N_{m,\varepsilon}^M}
L(p,q),
\end{equation}
where
\(
L(p,q)
:=2\mathfrak d_M(p,q) + \left| \KL(\pi_0^M\Vert p)-\KL(\pi_0^M\Vert q) \right|
\)
is the approximation loss.
Since
\[
\left|\KL(\pi_0^M\Vert p)
-\KL(\pi_0^M\Vert q) \right|
=\left|
\int_{A_M} 
\pi_0^M(x)\log\frac{q(x)}{p(x)}\,dx
\right|
\leq\mathfrak d_M(p,q),
\]
we have
\(\omega_m(\varepsilon,M)
\le 3\ve\).
Furthermore, Lemma~\ref{lem:ode-stability} gives, 
for any candidates \(\xi,\zeta\in\cX_m\),
\(
\mathfrak d_M(\pi_\xi^M,\pi_\zeta^M) \le C_{\mathrm{st}}(1+M)\norm{\xi-\zeta}_{\infty,[0,U]}
\)
with some \(C_{\mathrm{st}}=C(V,\kappa,K)\).
Combining this stability estimate with the network covering bound of Lemma~\ref{lem:dnn} yields
\begin{multline}
\mathcal H_m(\varepsilon,M)
\leq
\log N\left(
\frac{\varepsilon}{C_{\mathrm{st}}(1+M)},
\mathcal X_m,
\|\cdot\|_{\infty,[0,U]}
\right)\\
\leq
C_{\mathrm{ent}}m\log\left(
\frac{C_{\mathrm{ent}}C_{\mathrm{st}}
m(1+M)}{\varepsilon}
\right).
\end{multline}
We can now extend the finite-family concentration inequality to the full truncated sieve.
\begin{proposition}
\label{prop:truncated-oracle}
Assume \(Y_1,\ldots,Y_n\) are i.i.d. from \(\pi_0^M\) and the likelihood ratios satisfy, for all \(\xi\in\cX_m\), 
\(
 \gamma_M \le \pi_0^M/\pi_\xi^M \le \gamma_M^{-1}
\)
with some 
\(\gamma_M\in(0,1)\). If a data-dependent \(\widetilde\Xi_{n,M}\in\cX_m\) satisfies
\begin{equation}\label{eq:approx-truncated-optimizer}
\frac1n\sum_{i=1}^n\log\pi_{\widetilde\Xi_{n,M}}^M(Y_i)
 \geq
\sup_{\xi\in\cX_m}\frac1n\sum_{i=1}^n\log\pi_\xi^M(Y_i)-\varepsilon_{\mathrm{opt},M},
\end{equation}
then with probability at least $1-\delta$,
\begin{align}\label{eq:truncated-oracle}
\KL(\pi_0^M\Vert\pi_{\widetilde\Xi_{n,M}}^M)
 \leq
C_1\inf_{\xi\in\cX_m}\KL(\pi_0^M\Vert\pi_\xi^M)
 &+C_2\frac{B_{\gamma_M}(\mathcal H_m(\ve,M)+\log(2/\delta))}{n} \nonumber\\
 &\quad +C_3\omega_m(\ve,M)+C_4
 \varepsilon_{\mathrm{opt},M}
\end{align}
with some positive constants \(C_i\), \(i\in\{1,2,3,4\}\).
\end{proposition}
The proof is in Appendix~\ref{proof:prop:truncated-oracle}. 

The four terms in~\eqref{eq:truncated-oracle} represent, respectively, the sieve approximation error, the stochastic complexity, the discretization error and the numerical optimization error. We keep \(M\) and \(\varepsilon\) arbitrary at this stage. Their choices are made only after transferring~\eqref{eq:truncated-oracle} to the full likelihood and balancing the stochastic and truncation errors.

\subsection{Transfer to the full likelihood and control of truncation errors}\label{subsec:truncated-full}

To transfer to the full experiment, one needs to evaluate the loss of considering the truncated one. Define 
\(\bar\tau_m(M) :=\sup_{\xi\in\cX_m\cup\{\Xi_0\}}\tau_\xi(M)\), 
\(d_m(M) :=-\log(1-\bar\tau_m(M))\), 
and
\begin{equation}\label{eq:tail-log-remainder}
 \bar R_m(M) := \sup_{\xi\in\cX_m}\int_{A_M^c}\pi_0(x)\left|\log\frac{\pi_0(x)}{\pi_\xi(x)}\right|\,\dd x.
\end{equation}
The following lemma presents an exact decomposition of the full-KL divergence into the sum of truncated one and the quantities above.

\begin{lemma}
\label{lem:full-truncated-kl}
For every candidate $\xi$, and \(M>v_\star\),
\begin{eqnarray}\label{eq:exact-full-truncated-kl}
 \KL(\pi_0\Vert\pi_\xi)
 &=& (1-\tau_0(M))\KL(\pi_0^M\Vert\pi_\xi^M) + (1-\tau_0(M))\log\frac{1-\tau_0(M)}{1-\tau_\xi(M)} + R_\xi(M)\\
 &\leq& \KL(\pi_0^M\Vert\pi_\xi^M)+d_m(M)+\bar R_m(M),\label{eq:full-from-truncated}
\end{eqnarray}
where $R_\xi(M):=\int_{A_M^c}\pi_0(x) \log(\pi_0(x)/\pi_\xi(x))\,\dd x$. Moreover, 
\begin{equation}\label{eq:truncated-from-full}
\KL(\pi_0^M\Vert\pi_\xi^M)
\leq \frac{1}{c_M}\left(
\KL(\pi_0\Vert\pi_\xi)+d_m(M)+\bar R_m(M)
\right)
\end{equation}
with \(c_M:=\int_{A_M} e^{-K(V(x)-v_\star)}\,dx\left(\int_{\R^d} e^{-\kappa(V(x)-v_\star)}\,dx\right)^{-1}\in(0,1)\). 
\end{lemma}
The proof is in Appendix~\ref{proof:lem:full-truncated-kl}.

\begin{remark}
    For \(M\) large enough it holds that \(c_\infty/2\leq c_M\leq c_\infty\) with \(c_\infty:=Z_K/Z_\kappa\). Indeed, one observes that \(c_M\to c_\infty\) as \(M\to\infty\), while for any \(M\) fixed it holds 
    \[
    c_\infty-c_M
    =\frac{1}{Z_\kappa} 
    \int_{\{V(x)-v_\star>M-v_\star\}}
    e^{-K(V(x)-v_\star)}\,dx
    \leq \frac{Z_a}{Z_\kappa} e^{-(K-a)(M-v_\star)},
    \quad \forall\,a<K.
    \]
    Hence, \(c_\infty-Z_a Z_\kappa^{-1}e^{-(K-a)(M-v_\star)} \leq c_M\leq c_\infty\), with the left bound converging to \(c_\infty\) as \(M\) grows. In particular, the choice \(M_n\asymp \log n\) yields \(c_M\asymp c_\infty-O(n^{-\text{const} (K-a)})\). For \(M\) large enough the factor \(c_M^{-1}\) in~\eqref{eq:truncated-from-full} can thus be absorbed into a generic constant \(C_{V,\kappa,K}\).
\end{remark}

We now transfer the truncated oracle inequality~\eqref{eq:truncated-oracle} to the full MLE. For this,
we work on the no-escape event
\(
E_M:=\{X_1,\ldots,X_n\in A_M\},
\)
for which under the considered assumptions it holds that
\(
\mathbb P(E_M^c)\leq n\tau_0(M).
\)
On \(E_M\), the full and truncated likelihoods differ only through their
normalizing constants. Proposition~\ref{prop:full-mle-transfer} combines this observation with
the full--truncated KL decomposition of Lemma~\ref{lem:full-truncated-kl}.
\begin{proposition}
\label{prop:full-mle-transfer}
Let $X_1,\ldots,X_n$ be i.i.d. from $\pi_0$, and let $\widehat\Xi_n$ be the full MLE over $\cX_m$. Then with probability at least $(1-\delta)(1-\tau_0(M))^n\geq 1-\delta-n\tau_0(M)$ it holds that
\begin{align}\label{eq:full-mle-oracle-transfer}
 \KL(\pi_0\Vert\pi_{\widehat\Xi_n})
 &\leq \frac{C_1}{c_M}\inf_{\xi\in\cX_m}\KL(\pi_0\Vert\pi_\xi)
 +C_2\frac{B_{\gamma_M}(\mathcal H_m(\ve,M)+\log(2/\delta))}{n} \nonumber\\
 & +C_3\omega_m(\ve,M)
 +C_4\left(
 1+\frac{1}{c_M}
 \right)d_m(M)
 +C_5\left(
 1+\frac{1}{c_M}
 \right)\bar R_m(M)
\end{align}
with some constants \(C_i>0\), \(i\in\{1,...,5\}\), and \(c_M\) defined as in Lemma~\ref{lem:full-truncated-kl}.
\end{proposition}
The proof is in Appendix~\ref{proof:prop:full-mle-transfer}.

We can now choose the truncation level and net resolution. Let
\(
M_n=v_\star+C_{\mathrm{tr}}\log n
\)
with \(C_{\mathrm{tr}}\) sufficiently large
and
\(
\varepsilon_n=n^{-2},
\)
then
\(
\omega_m(\varepsilon_n,M_n)\leq3n^{-2}
\),
and
\(
\mathcal H_m(\varepsilon_n,M_n)\lesssim m\log n.
\)
Moreover, since as shown in Lemma~\ref{lem:explicit-tails} of Supplementary Material, 
\(d_m(M_n)\), \(\overline R_m(M_n)\), and \(n\tau_0(M_n)\) decay exponentially in \(M_n\), the corresponding terms become negligible. As by the same lemma \(B_{\gamma_M}=O(1)\) in the compact case, when no truncation \(M_n\) is needed, and \(B_{\gamma_M}\leq O(\log n)\) in the non-compact one, the leading stochastic term in the full oracle inequality is of order \(m\log n/n\) and \(m(\log n)^2/n\) in the compact and noncompact cases, respectively.

\section{Upper Rates: Exact Forward Approximation and Nonlinear Inversion}
\label{sec:upper-exact}

The oracle inequality of Proposition~\ref{prop:full-mle-transfer} reduces the statistical analysis of the MLE to two model-specific questions. First, one must evaluate the approximation term \(\inf_{\xi\in\cX_m}\KL(\pi_0\Vert\pi_\xi)\). Second, because the ultimate target is the coefficient \(\Xi_0\), one must determine how closeness of the invariant densities controls closeness of their coefficients. These are respectively the forward and inverse sides of the same nonlinear map \(\xi\mapsto\pi_\xi\). The present section studies that map in coordinates adapted to the stationary equation.

Define the inverse profile \(s_\xi(r):=q_\xi^{-1}(r)=\mu_\xi-g_\xi(r)\). It then holds that \(q_\xi(s_\xi(r))=r\) and \(s_\xi'(r)=-\rho_\xi(r)\) with \(\rho_\xi(r):=(r\xi(r))^{-1}\). These coordinates expose the basic geometry of the problem. In the forward direction, perturbations of \(\xi\) enter \(s_\xi\) through the primitive of \(\rho_\xi\); the forward map therefore smooths the coefficient perturbation by one integration. In the inverse direction, recovering \(\xi\) from the density requires differentiating the inverse profile, since \(\xi(r)=-(rs'_\xi(r))^{-1}\). This differentiation is unstable without regularity, which is why the inverse estimate will be governed by a H\"older interpolation modulus rather than by a uniform linear inverse inequality.

\subsection{Exact forward approximation and density rates}
We first use these exact coordinates to evaluate the approximation term left open in Section~\ref{sec:klbernstein}.

\begin{proposition}
\label{prop:exact-forward-approx}
Suppose that the assumptions of Lemma~\ref{lem:dnn} hold, and let \(\xi_m\), \(m\geq m_0\geq 1\), be the corresponding graded approximant. Define \(G_m(r)=g_{\xi_m}(r)-g_{\Xi_0}(r)\) 
and let \(\rho_t=(1-t)\rho_{\Xi_0}+t\rho_{\xi_m}\) and \(\xi_t(r)=(r\rho_t(r))^{-1}\) for \(0\leq t\leq1\). Finally, let \(\Pi_t\) be the stationary law for \(\xi_t\). Then it holds that
\begin{equation}
 \norm{G_m}_{\infty,[0,U]} =o(1), \quad
 \sup_{0\leq t\leq1}\norm{G_m}_{L^2(\nu_t)} \leq C_{\mathrm{path}} m^{-(\beta+1)},
\end{equation}
and, consequently, given~\eqref{eq:app_fwd}, 
\(\KL(\pi_0\Vert\pi_{\xi_m}) \leq C_{\mathrm{KL}} m^{-2(\beta+1)}\).
Here  
\(C_{\mathrm{path}},C_{\mathrm{KL}}\) are constants depending on \(\kappa\), K, as well as 
\(C_{\mathrm{app}},
C_{\mathrm{fwd}}\)
from Lemma~\ref{lem:dnn}, and independent of \(m\).
\end{proposition}
The proof is in Appendix~\ref{proof:prop:exact-forward-approx}.

Given this result and the discussion at the end of Subsection~\ref{subsec:truncated-full}, we can state the final rate for~\eqref{eq:full-mle-oracle-transfer}.

\begin{theorem}
\label{thm:oracle}
Assume that Assumptions~\ref{ass:potential} and~\ref{ass:coefficient} and the conditions of
Lemma~\ref{lem:dnn} hold, and
let $\widehat\Xi_n$ be the full MLE over the graded ReQU sieve \(\cX_m\) defined in \eqref{eq:smooth-dnn-sieve}. Fix \(c_a<\kappa\) and choose \(\ve>0\) such that \(\Xi_0(0)-\ve>c_a\). Finally, let \(M_n=v_\star+C_{tr}\log n\) with some \(C_{tr}>0\) such that \(C_{tr}c_a>1\).
Then with probability at least 
\(
(1-\delta)(1-\tau_0(M_n))^n
\)
it holds that
\begin{multline}\label{eq:oracle}
 \KL(\pi_0\Vert\pi_{\widehat\Xi_n})\\
 \leq C_\mathrm{or}
 \left( 
 m^{-2(\beta+1)} +\frac{b_n(m\log n+\log(1/\delta))}{n} +d_m(M_n)+\bar R_m(M_n) 
 \right),
\end{multline}
where
\(C_\mathrm{or}
=C(C_\mathrm{KL},
C_\mathrm{ent}, 
C_\mathrm{st},c_\infty,
C_\mathrm{tr})\) 
is a fixed constant, and
\(b_n=O(1)\) in the compact case and \(b_n\leq O(\log n)\) in the non-compact one.
Choosing \(m\asymp\left(n/(b_n\log n)\right)^{1/(2\beta+3)}\) yields
\begin{equation}\label{eq:kl-rate}
 \KL(\pi_0\Vert\pi_{\widehat\Xi_n}) = O_{\Prob}\left( (b_n\log n/n)^{2(\beta+1)/(2\beta+3)} \right).
\end{equation}
\end{theorem}

\begin{corollary}
\label{cor:finite-time-transfer}
The rates above can be transferred to the case of finite-time observations, provided that exponential ergodicity holds.
To observe that, let \(\V\) be a potential for which the corresponding nonlinear dynamics satisfy
\begin{equation}\label{eq:chi2-transfer-assumption}
\chi^2(P_t\Vert\Pi_0)
\leq
\eta_t
:=
C_{t_0}e^{-\rho_\chi(t-t_0)},
\qquad t\geq t_0,
\end{equation}
for some \(t_0\geq0\), \(C_{t_0}<\infty\), and \(\rho_\chi>0\), where \(P_t\) and \(\Pi_0\) denote the laws with densities \(p_t\) and \(\pi_0\), respectively. For the OU potential~\eqref{eq:ou-potential}, the inequality~\eqref{eq:chi2-transfer-assumption} holds under the assumptions of~\cite{belomestny-morozova-ou}, with
\(
C_{t_0}=\chi^2(P_{t_0}\Vert\Pi_0).
\)
Let \(X_1^{T_n},\ldots,X_n^{T_n}\) be independent observations from \(P_{T_n}\) for some \(T_n>0\), and let \(G_{M_n}\) be the event on which~\eqref{eq:full-mle-oracle-transfer} holds. Under the assumptions of Theorem~\ref{thm:oracle}, it holds that
\[
P_{T_n}^{\otimes n}(G_{M_n}^c)
\leq
\delta+n\tau_0(M_n)
+
\frac12\sqrt{(1+\eta_{T_n})^n-1}.
\]
Indeed,
\[
1+\chi^2(P_t^{\otimes n}\Vert\Pi_0^{\otimes n})
=
\left(1+\chi^2(P_t\Vert\Pi_0)\right)^{n}
\leq
(1+\eta_t)^n,
\]
and therefore
\(
\TV(P_t^{\otimes n},\Pi_0^{\otimes n})
\leq
0.5\sqrt{(1+\eta_t)^n-1}.
\)
The claimed bound follows from
\(
P_{T_n}^{\otimes n}(G_{M_n}^c)
\leq
\Pi_0^{\otimes n}(G_{M_n}^c)
+
\TV(P_{T_n}^{\otimes n},\Pi_0^{\otimes n})
\)
and
\(
\Pi_0^{\otimes n}(G_{M_n}^c)
\leq\delta+n\tau_0(M_n).
\)
Consequently, if
\(
n\eta_{T_n}\longrightarrow0
\)
and
\(
n\tau_0(M_n)\longrightarrow0,
\)
then the stationary oracle inequality and the rates in~\eqref{eq:kl-rate} remain valid for observations from \(P_{T_n}\). Under~\eqref{eq:chi2-transfer-assumption}, it is sufficient to choose, for any \(\varepsilon>0\),
\(
T_n
\geq
t_0+(1+\varepsilon)\rho_\chi^{-1}\log n,
\)
which guarantees
\(
n\eta_{T_n}
\leq
C_{t_0}n^{-\varepsilon}
\longrightarrow0
\)
as \(n\to\infty\).
Since by Lemma~\ref{lem:explicit-tails}
it holds that
\(
\tau_0(M)\lesssim e^{-c_a(M-v_\star)},
\)
taking
\(
M_n=v_\star+C_{\mathrm{tr}}\log n
\)
with \(c_aC_{\mathrm{tr}}>1\) yields the overall error \(\delta+o(1)\) and hence the desired rate.
\end{corollary}

\subsection{Nonlinear inverse stability and coefficient recovery}
We now turn to coefficient recovery. Although the boundedness of \(w_0\) on \(I=[r_-,r_+]\) suggests an inverse inequality of the form
\(
\|h\|_{L^2(I)}^2
\leq
Cm^2\|\cA h\|_{L^2(\nu_0)}^2,
\)
such a bound cannot hold uniformly over free-knot ReQU networks: the number of units does not control their smallest spatial scale. Indeed, for \(r_0\in\operatorname{int}(I)\), consider
\[
b_\delta(r)
=
(r-r_0)_+^2
-3(r-r_0-\delta)_+^2
+3(r-r_0-2\delta)_+^2
-(r-r_0-3\delta)_+^2.
\]
This four-unit ReQU network satisfies
\(
\|b_\delta\|_{L^2(I)}^2\asymp\delta^5
\)
and
\(
\|\cA b_\delta\|_{L^2(\nu_0)}^2\lesssim\delta^6,
\)
and hence the ratio
\(
\|b_\delta\|_{L^2(I)}^2/
\|\cA b_\delta\|_{L^2(\nu_0)}^2
\gtrsim\delta^{-1}
\)
diverges as \(\delta\to 0\).
Thus Proposition~\ref{prop:exact-forward-approx} cannot be inverted directly to control
\(
\|\widehat{\xi}_n-\Xi_0\|_{L^2(I)}.
\)
We instead derive coefficient recovery through control of \(s_\xi'\), using the interpolation result given in Lemma~\ref{lem:holder-besov-interpolation} of Supplementary Material alongside the following regular interior condition necessary to control the coefficient difference in terms of the KL-divergence.

\begin{assumption}\label{ass:regular-interior}
Fix a reference coefficient \(\xi_\circ\) satisfying Assumption~\ref{ass:coefficient}, and let \(I\subset (0,a_{\xi_\circ})\). We assume that there exists a compact interval \(S_I\subset (v_\star,\infty)\) disjoint from critical values of \(V\) such that \(s_{\xi_\circ}(I)\subseteq S_I\).
\end{assumption}
\begin{remark}\label{rem:coarea-bounds}
The endpoint bound~\eqref{eq:nu-zero-upper} controls the
large-energy regime \(s_0(r)\to\infty\) as \(r\downarrow0\). By
contrast, Assumption~\ref{ass:regular-interior} concerns the compact interior energy interval \(S_I\). Since \(S_I\) is disjoint from the critical values
of \(V\), Assumption~\ref{ass:potential} implies that there exist constants
\(0<c_{V,I}\leq C_{V,I}<\infty\) such that
\(
c_{V,I}\leq m_V(s)\leq C_{V,I}
\)
for all \(s\in S_I\).
\end{remark}
\begin{remark}
If Assumption~\ref{ass:regular-interior} is satisfied with \(\xi_\circ=\Xi_0\), the condition holds for the true energy interval \(s_0(I)\). In case when \(V\) has no critical values in \((v_\star,\infty)\), this assumption is automatically satisfied; this is the case, e.g., for the Ornstein-Uhlenbeck potential~\eqref{eq:ou-potential}.
\end{remark}

Provided the condition above is satisfied for the true coefficient \(\Xi_0\), one can establish the bound of the form 
\(
\norm{\xi-\Xi_0}_{L^2(I)} \lesssim
\KL(\pi_0\Vert\pi_\xi)^{\beta/(2(\beta+1))}
\)
for any candidate \(\xi\); see Lemmas~\ref{lem:exact-kl-modulus} and~\ref{lem:mode-localization} of the Supplementary Material. Given~\eqref{eq:oracle}, this allows to obtain the following unconditional bound for the coefficient difference.
\begin{theorem}
\label{thm:xi-upper}
Suppose that the assumptions of Theorem~\ref{thm:oracle} hold. Fix \(I=[r_-,r_+]\Subset(0,a_0)\) and suppose that Assumption~\ref{ass:regular-interior} holds on \(I\) for \(\xi_\circ=\Xi_0\). Let \(\mathfrak R_{n,m,\delta,M_n}\) denote the right-hand side of the KL-bound in~\eqref{eq:oracle}. There exists \(C_\mathrm{mode}=C(I,a_0,d,V,K)>0\) such that the inequality 
\(
\mathfrak R_{n,m,\delta,M_n}
\leq C_\mathrm{mode}
\) implies
\begin{equation}
 \Vert \widehat\Xi_n-\Xi_0\Vert_{L^2(I)} \leq C_\mathrm{inv} \mathfrak R_{n,m,\delta,M_n}^{\beta/(2(\beta+1))}
\end{equation}
for some constant 
\(
C_\mathrm{inv}:=C(I,a_0,V,\kappa,K,\beta,H,H_1)
\)
with probability at least \((1-\delta)(1-\tau_0(M_n))^n\).
Optimizing in $m$ gives
\begin{equation}\label{eq:xi-upper-rate}
 \Vert\widehat\Xi_n-\Xi_0\Vert_{L^2(I)} = O_{\Prob}\left( (b_n\log n/n)^{\beta/(2\beta+3)} \right).
\end{equation}
\end{theorem}
The proof is given in Appendix~\ref{proof:thm:xi-upper}.

\section{Minimax lower bounds for estimating \texorpdfstring{\(\Xi\)}{the drift coefficient}}\label{sec:lower}

We now show that the coefficient rate obtained in Section~\ref{sec:upper-exact} is minimax optimal, up to logarithmic factors. We work with a local parameter class around a fixed interior coefficient
\(\xi_\star\), and construct finitely many alternatives that are well
separated in \(L^2(I)\) but induce statistically close stationary
experiments. 
Lemma~\ref{lem:primitive-kl} first converts the KL divergence
between neighboring models into an \(L^2\)-bound for their primitive
difference. An Assouad cube of localized zero-mean perturbations then
yields the lower rate
\(
n^{-2\beta/(2\beta+3)}
\)
for the squared \(L^2(I)\)-risk.

\subsection{Local alternatives and KL geometry}

Fix a baseline coefficient $\xi_\star$ satisfying
\begin{equation}\label{eq:baseline-margin}
  \kappa+\epsilon_0\leq\xi_\star(r)\leq K-\epsilon_0,
  \qquad r\in[0,U],
\end{equation}
for some \(\epsilon_0>0\), with 
\(\xi_\star\in\mathcal H^\beta([0,U],H/2)\). 
Let \(a_\star\) be the maximal invariant density generated by \(\xi_\star\), and fix compact intervals
\(
  I=[r_-,r_+]\Subset J\Subset(0,a_\star).
\)
Finally, let 
\begin{multline}\label{eq:coeff_class_LB}
\mathfrak X_\beta(J)
:=
\left\{
\xi\in\mathcal H^\beta([0,U],H)\colon
\kappa\leq\xi(r)\leq K \, \forall\,r\in[0,U],
\right.\\\left. 
\xi(r)=\xi_\star(r)\ \text{for }r\in[0,U]\setminus J
\right\}
\end{multline}
the considered coefficient class. The following result holds.

\begin{lemma}
\label{lem:primitive-kl}
Let \(\xi_0,\xi_1\in \mathfrak X_\beta(J)\), and define the interpolations
\[
\rho_t:=\rho_{\xi_0}+t(\rho_{\xi_1}-\rho_{\xi_0})
\quad \text{and}\quad
\xi_t(r):=(r\rho_t(r))^{-1}, 
\quad 
0\leq t\leq 1.
\]
Suppose that Assumption~\ref{ass:regular-interior} holds on \(J\Subset(0,a_{\xi_\star})\) for \(\xi_\circ=\xi_\star\) and uniformly along the path \((\xi_t)_{t\in[0,1]}\), i.e., 
\(s_{\xi_t}(J)\subseteq S_J\) for all \(t\in[0,1]\).
Then, if 
\[
G(r):=\int_1^r (\rho_1(u)-\rho_0(u))\,du 
=0 
\quad\forall\, r\in[0,U]\setminus J,
\]
and
\[
\Vert G\Vert_\infty 
\leq \log 2/(K(1+C_{V,J}|J|\kappa^{-1}))
\]
with \(C_{V,J}\) the upper bound for \(m_V\) on \(S_J\) given by Remark~\ref{rem:coarea-bounds}, it holds that 
\begin{equation}\label{eq:primitive-kl-bound}
  \KL(\Pi_{\xi_1}\Vert \Pi_{\xi_0})
  \leq C_{J,\kappa,K,V}\norm{G}_{L^2(J)}^2,
\end{equation}
where \(C_{J,\kappa,K,V}\) is a uniform constant independent of the particular pair \((\xi_0,\xi_1)\) or \(G\).
\end{lemma}
The proof is given in Appendix~\ref{proof:lem:primitive-kl}.

\subsection{Assouad construction and minimax lower bound}

We now construct the family of local alternatives used in Assouad's
lemma. Let \(b>0\) be a bandwidth and choose
\(
M\asymp b^{-1}
\)
pairwise disjoint intervals of length comparable to \(b\) contained in
\(I\). On each interval, place a rescaled smooth bump \(\psi_j\) with
zero integral, and define
\(
\rho_\theta
=
\rho_\star
+
\alpha_b\sum_{j=1}^M\theta_j\psi_j
\)
with 
\(
\theta\in\{0,1\}^M,
\)
where \(\rho_\star\) corresponds to \(\xi_\star\), and the amplitude \(\alpha_b\asymp b^\beta\) is chosen so that the
corresponding coefficients \(\xi_\theta\) remain in \(\mathfrak X_\beta(I)\). The
zero-mean condition localizes the associated primitive perturbations,
while disjoint supports make the \(L^2(I)\)-separation additive over the
coordinates of the cube. For neighboring vertices
\(\theta\) and \(\theta^{(j)}\), the coefficient separation and the
one-observation KL divergence have orders
\[
\|\xi_\theta-\xi_{\theta^{(j)}}\|_{L^2(I)}^2
\asymp b^{2\beta+1},
\qquad
\operatorname{KL}
\bigl(\Pi_{\xi_\theta}\|\Pi_{\xi_{\theta^{(j)}}}\bigr)
\lesssim b^{2\beta+3}.
\]
Thus, taking \(b\asymp n^{-1/(2\beta+3)}\) keeps the KL divergence between every pair of neighboring \(n\)-sample
experiments bounded by a constant and leads, through Assouad's
lemma, to the minimax lower bound stated below.

\begin{theorem}
\label{thm:minimax-lower}
Under the same assumptions as in Lemma~\ref{lem:primitive-kl}, there exists a constant \(C_\mathrm{risk}\) depending on \(I,J,\beta,H,\kappa,K,\xi_\star,d\) and \(V\), but not on \(n\), such that
\begin{equation}\label{eq:lower-risk}
  \inf_{\wh\Xi_n}
  \sup_{\xi\in\mathfrak X_\beta(J)}
  \E_\xi\norm{\wh\Xi_n-\xi}_{L^2(I)}^2
  \geq C_\mathrm{risk} n^{-2\beta/(2\beta+3)}.
\end{equation}
Moreover, for constants 
\(C_\mathrm{sep}>0\) and \(p\in(0,1)\)
depending on the same variables as \(C_\mathrm{risk}\),
\begin{equation}\label{eq:lower-prob}
  \inf_{\wh\Xi_n}
  \sup_{\xi\in\mathfrak X_\beta(J)}
  \Pi_\xi^{\otimes n}\!\biggl(
    \norm{\wh\Xi_n-\xi}_{L^2(I)}
    \geq C_\mathrm{sep} n^{-\beta/(2\beta+3)}
  \biggr)
  \geq p.
\end{equation}
The infimum is over all measurable estimators based on \(X_1,\ldots,X_n\sim \Pi_\xi\).
\end{theorem}
The proof is deferred to Appendix~\ref{proof:thm:minimax-lower}.



\begin{remark}
Theorem~\ref{thm:minimax-lower} concerns the stationary experiment. 
The transfer to finite-time observations is analogous to the product-law argument of Corollary~\ref{cor:finite-time-transfer}, but must hold uniformly over the Assouad cube
\(
\mathcal C_n:=\{\xi_\theta:\theta\in\{0,1\}^{M_n}\}.
\)
Writing \(P_{\xi,T}\) and \(\Pi_\xi\) for the finite-time and stationary laws, respectively, it is sufficient that
\(
\Delta_{n,T}
:=
\sup_{\xi\in\mathcal C_n}
\TV\!\left(P_{\xi,T}^{\otimes n},\Pi_\xi^{\otimes n}\right)
\to0.
\)
Indeed, for neighboring vertices \(\theta,\theta^{(j)}\),
\(
\TV\!\left(
P_{\xi_\theta,T}^{\otimes n},
P_{\xi_{\theta^{(j)}},T}^{\otimes n}
\right)
\leq
\TV\!\left(
\Pi_{\xi_\theta}^{\otimes n},
\Pi_{\xi_{\theta^{(j)}}}^{\otimes n}
\right)
+2\Delta_{n,T},
\)
so the testing separation used in Assouad's lemma is preserved. As in Corollary~\ref{cor:finite-time-transfer}, a uniform bound
\[
\sup_{\xi\in\mathcal C_n}
\chi^2(P_{\xi,T}\|\Pi_\xi)
\leq Ce^{-\rho T}
\]
therefore permits \(T_n\gtrsim \rho^{-1}\log n\). Alternatively, one may prove the required neighboring finite-time KL bound directly.
\end{remark}

\section{Simulation study}
\label{sec:sim-study}

For the simulation study, we consider the OU potential~\eqref{eq:ou-potential} alongside the following three drift coefficients:
\begin{equation}
\Xi_{0,1}(r)=1,
\quad
\Xi_{0,2}(r)=0.8+0.4(
0.5-0.1(1-r)_+^2
),
\quad
\text{and}
\quad
\Xi_{0,3}(r)=0.8+0.4r/(1+r).
\end{equation}
The first function presents a prime example of the Ornstein-Uhlenbeck dynamics; the choice of constants for the remaining two ensures that the theoretical assumptions of~\cite{belomestny-morozova-ou} are satisfied, assuming for simplicity that the initial distribution is given by a Gaussian law with covariance matrix \(\sigma^2 I\), \(\sigma^2<0.5\), and taking \(\kappa=0.8\), \(K=1.2\) for \(\Xi_{0,1}\), \(\Xi_{0,2}\), and \(\kappa=0.7\), \(K=1.3\) for \(\Xi_{0,3}\). 
The sieve is defined as in~\eqref{eq:smooth-dnn-sieve} with 
\(\tau=0.05\)
and
\(U=\max\{1,(K/\pi)^{d/2}\}=1\).
For \(q\geq1\), let \(0<t_1<\cdots<t_q\leq U\) be fixed knots and parameterize the network by
\[
N_{\theta,q}(r)
=
a_1+a_2r+
\sum_{j=1}^{q}
a_{j+2}
\left\{
t_j^2-(t_j-r)_+^2
\right\},
\]
subject to \(a_1>\tau\), 
\(a_2,a_3,\ldots,a_{q+2}\in[0,1]\)
and
\(
a_1+a_2U+
\sum_{j=1}^{q}a_{j+2}t_j^2
<
1-\tau,
\)
which guarantees that
\(
\tau<N_{\theta,q}(0)
\leq
N_{\theta,q}(r)
\leq
N_{\theta,q}(U)
<
1-\tau,
\)
so \(N_{\theta,q}\) remains in the identity region of \(S_\tau\) on \([0,U]\). The linear term requires two ReQU units and each knot contributes one quadratic-hinge unit; hence the family has an explicit realization in
\(
\operatorname{NN}
\left(
2,(1,q+2,1),3q+7,1
\right).
\)
We set \(t_q=U\), so that \(q=1\) gives
\(
N_{\theta,1}(r)
=
a_1+a_2r+a_3
\left(
U^2-(U-r)_+^2
\right).
\)
This smallest family already contains \(\Xi_{0,1}\) and \(\Xi_{0,2}\), whereas additional knots for \(q>1\) provide greater flexibility for approximating the non-polynomial coefficient \(\Xi_{0,3}\) that does not belong to the sieve.

\begin{figure}[t]
\centering
\includegraphics[width=\linewidth]{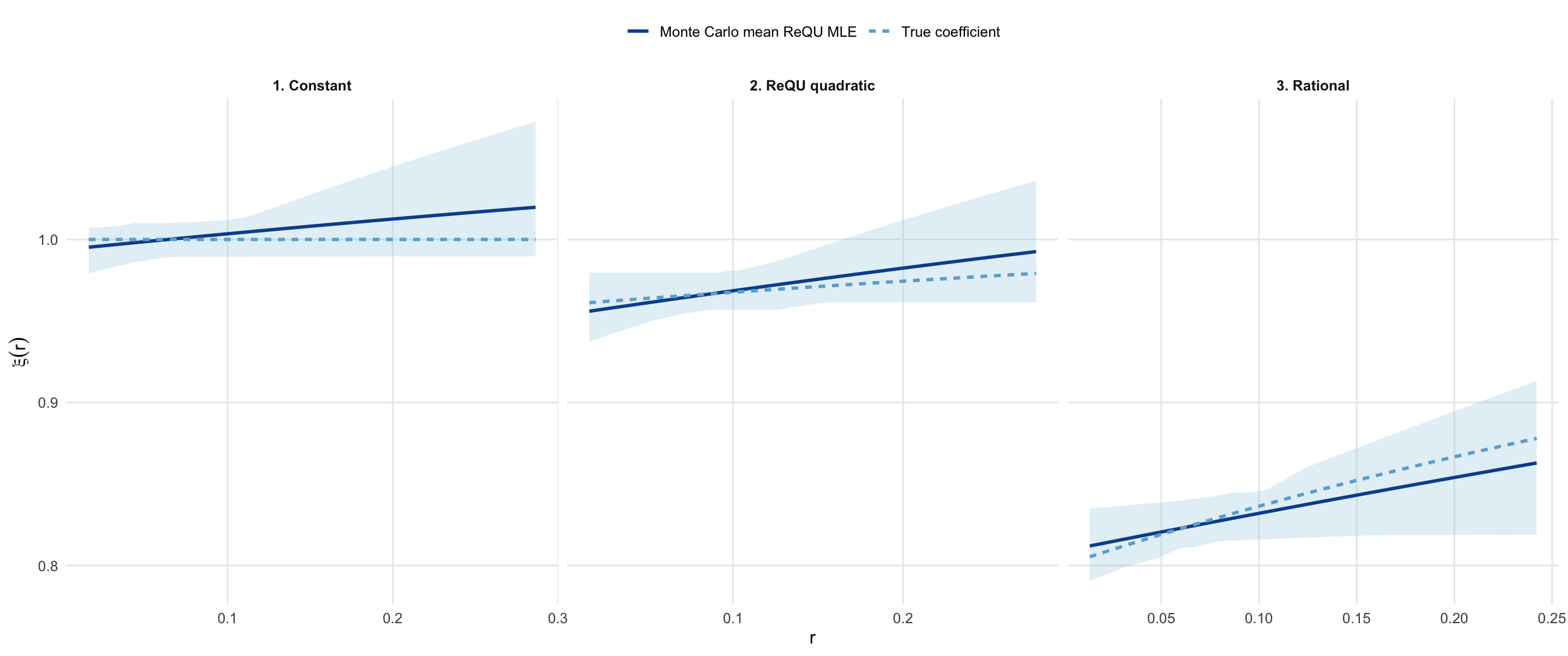}
\caption{True interaction coefficients (dashed), their ReQU-sieve
maximum-likelihood estimates (solid) and the numerical confidence intervals (blue) over 20 simulation runs for \(n=10000\) and \(d=2\).}
\label{fig:xi-estimates}
\end{figure}

\begin{table}[t]
\centering
\small
\caption{The errors \(ISE_{n,i,d,1}\) for \(n\in\{100,1000,10000\}\) and \(d\in\{1,\dots,5\}\) averaged over \(100\) repetitions.}
\label{tab:xi-mse}
\begin{tabular}{lc|ccccc}
\hline
\multicolumn{2}{c|}{Coefficient} & $d=1$ & $d=2$ & $d=3$ & $d=4$ & $d=5$\\
\hline
\multirow{3}{*}{\(\Xi_{0,1}\)} 
& \(n=100\) & 0.01442 & 0.00844 & 0.00626 & 0.0048 & 0.00299\\
& \(n=1000\) & 0.00356 & 0.00137 & 0.00083 & 0.00057 & 0.00041 \\
& \(n=10000\) & 0.00052 & 0.00033 & 0.00019 & 0.0001 & 0.00007\\ \hline
\multirow{3}{*}{\(\Xi_{0,2}\)}
& \(n=100\) & 0.01317 & 0.00918 & 0.00518 & 0.00405 & 0.00425\\
& \(n=1000\) & 0.00283 & 0.00151 & 0.00087 & 0.00047 & 0.0004 \\
& \(n=10000\) & 0.00057 & 0.00031 & 0.00024 & 0.00013 & 0.00007\\ \hline
\multirow{3}{*}{\(\Xi_{0,3}\)}
& \(n=100\) & 0.01243 & 0.00734 & 0.00508 & 0.00307 & 0.00198\\
& \(n=1000\) & 0.00306 & 0.00172 & 0.00079 & 0.00046 & 0.00037 \\
& \(n=10000\) & 0.00048 & 0.00039 & 0.00023 & 0.00017 & 0.00006 \\
\hline
\end{tabular}
\end{table}

\begin{figure}[t]
\centering
\includegraphics[width=\linewidth]{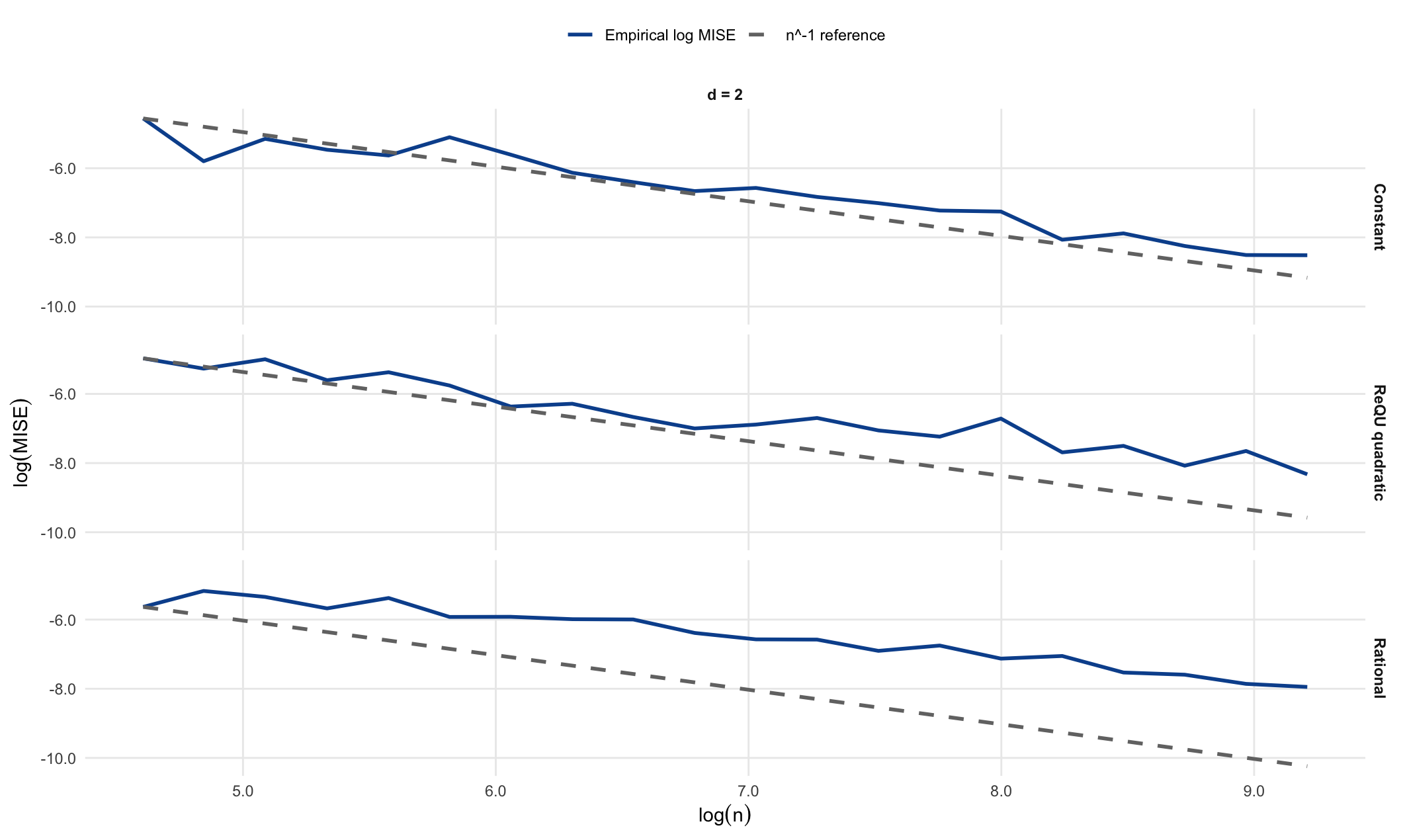}
\caption{The errors \(ISE_i^{(n,2)}\) averaged over 20 simulation runs and the benchmark \(n^{-1}\)-rate, log scale.}
\label{fig:rates20}
\end{figure}

In what follows, we are considering samples from the stationary density \(\pi_{0,i}\) corresponding to the coefficient \(\Xi_{0,i}\), \(i=1,2,3\), generated by rejection sampling. More precisely, we note that \(X\sim \pi_0\) is equal in distribution to \(\sqrt{T}Z/\Vert Z\Vert_2\) with \(Z\) a standard normal random variable and \(T\) having the density \((\pi^{d/2}/\Gamma(d/2))q_0(1+t)t^{d/2-1}\), \(t>0\), where the latter variable is simulated using acceptance-rejection algorithm with Gamma distribution having shape \(d/2\) and rate \(\kappa\). Given such sample, we follow the algorithm described in Subsection~\ref{subsec:mle} to estimate \(\widehat{\Xi}_{n,i}\). 
The value \(q\) in the network parametrisation is chosen by minimising the grid error 
\[ 
\operatorname{ISE}_{n,i,d,q}
=
\frac{1}{G}
\sum_{g=1}^{G}
\left\{
\widehat{\Xi}_{n,i,q}(r_g)
-\Xi_{0,i}(r_g)
\right\}^2,
\quad 
q\in\{1,2,3\},
\quad 
G=501,
\]
averaged over 20 simulation runs with samples of size \(n=10000\) and \(d=2\), 
where \(\widehat{\Xi}_{n,i,q}\) is the estimator obtained using the network with 
\(r_g:=0.05 a_{0,i}+(g-1)(0.9 a_{0,i}-0.05 a_{0,i})/(G-1)\) and \(a_{0,i}=q_{0,i}(1)\). Since the grid is based on the interval \(I_{0,i}:=[0.05 a_{0,i}, 0.9 a_{0,i}]\), it excludes 
the endpoints \(r=0\) and \(r=a_{0,i}\), where estimation is less stable: very small density values correspond to sparsely observed tails, while the exact upper endpoint corresponds only to the mode.
The resulting optimal value is \(q=1\), i.e., the chosen network is \(\NN(2,(1,3,1),10,1)\). 

Figure~\ref{fig:xi-estimates} presents the averaged estimator 
\(\widehat{\Xi}_{n,i,1}\) obtained using this optimal paramet-risation over 20 repetitions alongside the corresponding numerical confidence interval. 
It can be seen that the estimated and true curves are fairly close, and the errors do not exceed \(0.05\). 
This observation is further supported by Table~\ref{tab:xi-mse} containing the values of \(ISE_{n,i,d,1}\) obtained for different sample sizes \(n\in\{100,1000,10000\}\) and dimensions \(d\in\{1,\dots,5\}\) and averaged over \(100\) simulation runs. The resulting values do not exceed \(0.02\) even for the smallest number of observations, and the errors decay rather fast as the sample size grows. 
Finally, we analyse the rate of convergence of the constructed estimator \(\widehat{\Xi}_{n,i}\) to the true coefficient \(\Xi_{0,i}\). Figure~\ref{fig:rates20} depicts the error \(ISE_{n,i,2,1}\) obtained for different values of \(n\) from \(100\) to \(10000\) and averaged over \(20\) repetitions, on a log scale. 
It can again be seen that the error decreases with the growth of \(n\). Since \(\Xi_{0,1}\) and \(\Xi_{0,2}\) belong to the constructed sieve, one would expect the decay to be of order \(n^{-1}\) in these cases, while for \(\Xi_{0,3}\) it should be slower due to the presence of approximation error. As can be observed from the curves in Figure~\ref{fig:rates20}, in case of \(\Xi_{0,1}\) the rate of error decay is indeed approximately \(n^{-1}\), though for \(\Xi_{0,2}\) this is less evident. To the contrary, in case of \(\Xi_{0,3}\) the error converges to zero at a noticeably slower rate.
\section*{Funding}
The present research is supported by the Deutsche Forschungsgemeinschaft through the BE 3961/7-1 ``Statistische Inferenz f\"ur Teilchensysteme und McKean--Vlasov-SDEs mit singul\"aren Kernen''

\appendix
\section{Proofs}
\label{sec:proofs}
\subsection{Proof of Proposition~\ref{prop:invariant-general}}\label{proof:prop:invariant-general}
Due to Assumption~\ref{ass:coefficient}, the function \(g_\xi\) defined in~\eqref{eq:gxi} is a bijection with a \(C^1\)-inverse \(g_\xi^{-1}(z)\in (0,\infty)\) for all \(z\in\R\).
For \(\mu\in\mathbb R\), define
\(
\pi_{\xi,\mu}(x)
:=g_\xi^{-1}\!\left(\mu-\V(x)\right)
\)
and
\(
F_\xi(\mu)
:=
\int_{\mathbb R^d}\pi_{\xi,\mu}(x)\,\dd x.
\)
The coefficient bounds in Assumption~\ref{ass:coefficient} imply that
\(
g_\xi^{-1}(z)\leq e^{\kappa z}
\)
for \(z\leq0\).
Together with the coercivity of \(\V\) and the integrability of \(e^{-\kappa\V}\) guaranteed by Assumption~\ref{ass:potential}, this yields that \(F_\xi(\mu)<\infty\). Moreover, \(F_\xi\) is continuous and strictly increasing, with \(F_\xi(\mu)\to 0\) as \(\mu\to-\infty\) and \(F_\xi(\mu)\to\infty\) as \(\mu\to\infty\).
Consequently, there exists a unique \(\mu_\xi\in\mathbb R\) such that
\(
F_\xi(\mu_\xi)=1.
\)
Setting \(\pi_\xi\) as in~\eqref{eq:pi-general}, we get that \(\pi_\xi>0\) due to the range of \(g^{-1}_\xi\), and \(\pi_\xi\in C^1\) due to \(\V\in C^2\) and \(g_\xi^{-1}\in C^1\). Differentiation gives
\(
\nabla\pi_\xi(x)
=
-\pi_\xi(x)\xi(\pi_\xi(x))\nabla\V(x),
\)
so \(\pi_\xi\) satisfies the zero-flux equation. It remains to prove uniqueness. Let \(\widetilde\pi\) be another nonnegative \(C^1\) zero-flux stationary density. First note that \(\widetilde\pi\) is strictly positive. Indeed, if \(\widetilde\pi(x_0)=0\) for some \(x_0\in\R^d\), then, choosing for any \(x\in\mathbb R^d\) a \(C^1\) path \(\gamma\) joining \(x_0\) to \(x\) and setting
\(
y(t)=\widetilde\pi(\gamma(t)),
\)
we get from the zero-flux equation
\(
y'(t)
=
-y(t)\xi(y(t))
\nabla\V(\gamma(t))\cdot\gamma'(t).
\)
Since \(\xi\leq K\) and 
\(
t\mapsto\nabla\V(\gamma(t))\cdot\gamma'(t)
\) 
is continuous on \([0,1]\) and therefore bounded, 
there exists \(C_\gamma<\infty\) such that
\(|y'(t)|\leq C_\gamma |y(t)|.\) If \(y(0)=0\), Gr\"onwall's inequality implies \(y\equiv0\) along the path, implying that \(\widetilde\pi(x)=0\) for every \(x\) which contradicts its normalization. Thus, \(\widetilde\pi>0\), and since \(g'_\xi(r)=(r\xi(r))^{-1}\), the zero-flux condition can equivalently be rewritten as 
\(
\nabla(
g_\xi(\widetilde\pi(x))+\V(x)
)
=0.
\)
Since \(\mathbb R^d\) is connected, there exists \(\widetilde\mu\in\R\) such that
\(
g_\xi(\widetilde\pi(x))+\V(x)=\widetilde\mu\),
from which
\(
\widetilde\pi(x)
=
g_\xi^{-1}\!\left(\widetilde\mu-\V(x)\right).
\)
Normalization again implies \(F_\xi(\widetilde\mu)=1\). By uniqueness of the normalizing constant, \(\widetilde\mu=\mu_\xi\), and hence \(\widetilde\pi=\pi_\xi\).

\subsection{Proof of Lemma~\ref{lem:dnn}}\label{proof:lem:dnn}
Let \(k=\lceil\beta\rceil-1\) and \(\alpha_\beta=\beta-k\in(0,1]\). 
Since \(\Xi_0\in\cH^\beta\), it follows that \(\Xi_0\) has \(k\) derivatives and \(\Xi_0^{(k)}\) is \(\alpha_\beta\)-H\"older continuous. The aligned graded mesh satisfies, up to constants independent of \(j\) and \(m\), 
\(
r_{j,m}\asymp \left(j/m\right)^\vartheta,
\)
\(
\Delta_{j,m}:=r_{j,m}-r_{j-1,m}
\asymp m^{-\vartheta}j^{\vartheta-1}
\)
and \(\max_j\Delta_{j,m}\lesssim m^{-1}\).
We first construct a piecewise-polynomial approximation. Define the Hermite polynomial 
\(H_k(t)
= (\int_0^t z^k(1-z)^k\,\dd z)
/(\int_0^1 z^k(1-z)^k\,\dd z)\) 
for \(0\leq t\leq 1\).
On \(J_{j,m}=[r_{j-1,m},r_{j,m}]\), let
\(
t=(r-r_{j-1,m})\Delta_{j,m}^{-1}
\)
and set
\(
(Q_m\Xi_0)(r)
=
(1-H_k(t))(T_{r_{j-1,m}}\Xi_0)(r)
+
H_k(t)(T_{r_{j,m}}\Xi_0)(r),
\)
where \(T_x\Xi_0\) denotes the degree-\(k\) Taylor polynomial at \(x\). Since \(H_k\) and \(1-H_k\) vanish to order \(k+1\) at the appropriate endpoints, the pieces and their derivatives through order \(k\) agree at every knot. Taylor's theorem gives
\begin{equation}\label{eq:appr_deriv_bound}
\Vert
(Q_m\Xi_0-\Xi_0)^{(\ell)}
\Vert_{L^\infty(J_{j,m})}
\leq
C_\beta H\Delta_{j,m}^{\beta-\ell},
\qquad
0\leq\ell\leq k
\end{equation}
with some \(C_\beta>0\).

This approximation has to be corrected because errors in the primitive operator
\(
(\cK h)(r)
=
\int_1^r h(u)w_0(u)\,\dd u
\)
with
\(
w_0(u)=(u\Xi_0(u)^2)^{-1},
\)
could otherwise accumulate across cells. Define the polynomial bubbles \(B_{j,m}(r)= \Delta_{j,m}^{\beta} \phi_k((r-r_{j-1,m})/\Delta_{j,m}) \1_{J_{j,m}}(r)\) with \(\phi_k(t)=t^{k+1}(1-t)^{k+1}\), and set the approximant \(\xi_m\) to be the corrected spline 
\begin{equation}\label{eq:spline_def}
    \xi_m=Q_m\Xi_0+\sum_{j=1}^m c_{j,m}B_{j,m}.
\end{equation}
Choose \(c_{j,m}\) so that~\eqref{eq:approximant-moments} holds, which, given that on the cell \(J_{j,m}\) only \(B_{j,m}\) is non-zero, is equivalent to choosing
\[
c_{j,m}
=
-
\frac{
\int_{J_{j,m}}
(Q_m\Xi_0-\Xi_0)(u)w_0(u)\,\dd u
}{
\int_{J_{j,m}}
B_{j,m}(u)w_0(u)\,\dd u
}
=:\frac{E_{j,m}}{D_{j,m}}.
\]
These coefficients are uniformly bounded. Indeed, on the first cell, the endpoint interpolation gives
\[
|Q_m\Xi_0(u)-\Xi_0(u)|
\lesssim u^\beta,
\qquad 0\leq u\leq\Delta_{1,m},
\]
and, since \(\kappa\leq\Xi_0\leq K\),
it follows that 
\(|E_{j,m}|\lesssim
\Delta_{1,m}^{\beta}\),
whereas
\[
\int_0^{\Delta_{1,m}}
B_{1,m}(u)w_0(u)\,\dd u
=
\Delta_{1,m}^{\beta}
\int_0^1
\frac{\phi_k(t)}
{t\Xi_0(\Delta_{1,m}t)^2}\,\dd t
\gtrsim
\Delta_{1,m}^{\beta},
\]
implying \(|c_{1,m}|\lesssim1\). For \(j\geq2\), the weight \(w_0\) is comparable on \(J_{j,m}\), and it holds that
\(
|E_{j,m}|
\lesssim \Delta_{j,m}^\beta 
\int_{J_{j,m}} w_0(u)\,du
\)
and
\(
D_{j,m} \gtrsim \Delta_{j,m}^\beta 
\int_{J_{j,m}} w_0(u)\,du
\).
Hence \(\sup_{j,m}|c_{j,m}|<\infty\).
The bubbles and their derivatives through order \(k\) vanish at the cell endpoints and satisfy
\(
\|B_{j,m}^{(\ell)}\|_{L^\infty(J_{j,m})}
\lesssim
\Delta_{j,m}^{\beta-\ell}.
\)
Together with~\eqref{eq:appr_deriv_bound}, this shows that \(\xi_m\) is globally \(\beta\)-H\"older and 
\(
\|\xi_m\|_{\cH^\beta([0,U])}\leq H_1
\)
for a fixed \(H_1\) independent of \(m\).
Moreover, since 
\(\|\xi_m-\Xi_0\|_{\infty,[0,U]}
\lesssim m^{-\beta}\)
for \(m\) large enough, the interior condition~\eqref{eq:truth-interior-envelope} on \(\Xi_0\) gives
\(
\kappa+\eta_0/2
\leq
\xi_m
\leq
K-\eta_0/2.
\)
Consequently,
\(
(\xi_m-\kappa)/(K-\kappa)
\)
lies in the identity region of \(S_\tau\), so the clamping operation leaves \(\xi_m\), and hence the cancellations~\eqref{eq:approximant-moments}, unchanged.

We next control the forward approximation error. Write
\(
e_m:=\xi_m-\Xi_0.
\)
By~\eqref{eq:approximant-moments}, \((\cK e_m)(r_{j,m})=0\) at every knot. On the first cell this gives 
\[(\cK e_m)(r)
=
\int_0^r e_m(u)w_0(u)\,\dd u,
\quad 
0\leq r\leq\Delta_{1,m},
\]
and hence
\(
|(\cK e_m)(r)|
\lesssim r^\beta.
\)
Using~\eqref{eq:nu-zero-upper},
\begin{equation*}
\int_0^{\Delta_{1,m}}
|(\cK e_m)(r)|^2\nu_0(r)\,dr
\lesssim C_\nu\Delta_{1,m}^{2\beta+1}
\left(
1+\log\frac{r_0}{r}
\right)^\alpha
=O(m^{-2(\beta+1)}),
\end{equation*}
the last equality being due to \(\Delta_{1,m}\asymp m^{-\theta}\) and \(\theta(2\beta+1)>2\beta+2\).
For \(j\geq2\) and \(r\in J_{j,m}\), moment cancellation similarly yields
\(
|(\cK e_m)(r)|
\lesssim
\Delta_{j,m}^{\beta+1}r_{j-1,m}^{-1}.
\)
For the cells contained in \((0,r_0]\), the graded-mesh estimates and~\eqref{eq:nu-zero-upper} give
\begin{multline}
\sum_{\substack{j\geq2\\r_{j,m}\leq r_0}}
\int_{J_{j,m}}
|(\cK e_m)(r)|^2v_0(r)\,\dd r
\lesssim
m^{-\vartheta(2\beta+1)}
\sum_j
j^{\vartheta(2\beta+1)-(2\beta+3)}
\left(1+\log\frac{Cm}{j}\right)^{\alpha}\\ \lesssim
m^{-2(\beta+1)},
\end{multline}
the last inequality following because
\(
\vartheta(2\beta+1)-(2\beta+3)>-1.
\)
Away from zero, \(w_0\) is bounded and \(\Delta_{j,m}\lesssim m^{-1}\). Therefore,
\(
\sup_{r\geq r_0}|(\cK e_m)(r)|
\lesssim
m^{-(\beta+1)},
\)
and since \(\nu_0\) is a probability measure, the same order holds in \(L^2(\nu_0)\). Combining the two regions gives
\(
\|\cK e_m\|_{L^2(\nu_0)}
\lesssim
m^{-(\beta+1)}.
\)
Since by~\eqref{eq:c0-functional} 
\(
|\dot\mu_0(e_m)|
\lesssim
\|\cK e_m\|_{L^2(\nu_0)},
\)
we conclude that 
\(
\|\cA e_m\|_{L^2(\nu_0)}
\lesssim
m^{-(\beta+1)}.
\)

It remains to verify the network complexity. The function \(\xi_m\) is piecewise polynomial on \(m\) cells, of degree at most
\(
D_\beta=\max\{3k+1,2k+2\},
\)
which is independent of \(m\). Standard exact ReQU realizations of fixed-degree piecewise polynomials therefore give a network of fixed depth with \(O(m)\) nonzero parameters. Since the smallest cell has polynomial size in \(m^{-1}\), all rescaling weights are polynomially bounded. Thus \(\xi_m\in\cX_m\).
Finally, the full sparse-network sieve has \(O(m)\) active parameters, polynomially many available parameters and polynomial weight bounds. The standard sparse-network covering estimate consequently gives
\(\log N(\varepsilon,\cX_m,\norm{\cdot}_{\infty,[0,U]}) \leq C_{\mathrm{ent}} m\log(Cm/\varepsilon)\) over the full free-network sieve,
which completes the proof.

\subsection{Proof of Lemma~\ref{lem:kl-bernstein}}
\label{proof:lem:kl-bernstein}
For a given \(p^A\in\cP_A\), let \(Z_{p^A}(Y_1):=\log(p_0^A(Y_1)/p^A(Y_1))\). Denoting also \(R(y):=p_0^A(y)/p^A(y)\), one can observe that, since \(\E_{p_0^A}[1/R]=1\), it holds that 
\(\E_{p_0^A}[\log R(Y_1)-1+1/R(Y_1)]
=\E_{p_0^A}[Z_{p^A}(Y_1)]
=\KL(p_0^A\Vert p)\). 
Since 
\(
(\log r)^2 \leq B_\gamma (\log r-1+1/r)
\)
for any \(\gamma\leq r\leq \gamma^{-1}\)
with \(B_\gamma\) defined in~\eqref{eq:B-gamma}, it further follows that 
\[
\var(Z_{p^A}(Y_1))
\leq\E_{p_0}[Z_{p^A}^2(Y_1)]
\leq B_\gamma \KL(p_0^A\Vert p^A).
\]
Applying the Bernstein inequality to 
\(|X_{i,p^A}|
:=|Z_{p^A}(Y_i)-\E_{p_0^A}[Z_{p^A}(Y_i)]|
\leq 2\Lambda_\gamma\),
we get that, for any \(t\geq 0\),
\[
\left|
(P_n-P_0)Z_{p^A}
\right|
>\sqrt{\frac{2B_\gamma \KL(p_0^A\Vert p^A)t}{n}}
+\frac{2\Lambda_\gamma t}{3n}
\]
with probability no greater than \(2e^{-t}\). By the union bound, the probability that at least one \(p^A\in\cP_A\) fails is then upper bounded by \(2|\cP_A|e^{-t}\). Choosing \(t:=\log|\cP_A|+\log(2/\delta)\), we conclude~\eqref{eq:kl-bernstein}, and by the Young's inequality~\eqref{eq:kl-bernstein-linearized} follows.

\subsection{Proof of Proposition~\ref{prop:truncated-oracle}}
\label{proof:prop:truncated-oracle}
Let \(p^*\in\mathcal{P}_m^M\) be an oracle minimising \(K(p)=\KL(\pi_0^M\Vert p)\); if the infimum is not attained, choose \(p^*\in\mathcal{P}_m^M\) such that 
\(K(p^*)
\leq \inf_{p\in\mathcal{P}_m^M}K(p)+e_{p^*}
\)
with some error \(e_{p^*}\to 0\). By~\eqref{eq:approx-truncated-optimizer},
\[
P_n\log\widehat{p}
\geq P_n\log p^* - \varepsilon_{\mathrm{opt},M}
\Leftrightarrow 
K(\widehat{p})
\leq K(p^*)
+\Delta_n(p^*)
-\Delta_n(\widehat{p})
+\varepsilon_{\mathrm{opt},M},
\]
where \(\Delta_n(p):=(P_n-P_0)\log(p_0/p)\). While Lemma~\ref{lem:kl-bernstein} provides a useful bound for \(\Delta_n(p)\), it can only be applied in case when \(p\) belongs to some finite set, while \(\mathcal{P}_m^M\) can be infinite. To this end, choose representatives \(q^*(p^*)\), \(\widehat{q}(\widehat{p})\in \mathcal{N}_{m,\ve}^M\), where the latter set is finite. Since for all \(p,q\) it holds that 
\[
|\Delta_n(p)-\Delta_n(q)|
\leq \left|P_n\log\frac{q}{p}\right|
+\left|P_0\log\frac{q}{p}\right|
\leq \Vert \log\frac{q}{p}\Vert_{\infty,\,A_M}
\leq 2\mathfrak d_M(p,q),
\]
it follows that 
\[
\Delta_n(p^*)
\leq \Delta_n(q^*)
+2\mathfrak d_M(p^*,q^*)
\quad\text{and}\quad
-\Delta_n(\widehat{p})
\leq -\Delta_n(\widehat{q})
+2\mathfrak d_M(\widehat{p},\widehat{q}),
\]
from which, applying Lemma~\ref{lem:kl-bernstein},
\begin{multline*}
K(\widehat{p})
\leq K(p^*)
+\Delta_n(q^*) - \Delta_n(\widehat{q})
+2\mathfrak d_M(p^*,q^*)
+2\mathfrak d_M(\widehat{p},\widehat{q})
+\varepsilon_{\mathrm{opt},M}\\
\leq (1+\eta)K(p^*)
+\eta K(\widehat{p})
+\eta|K(q^*)-K(p^*)|
+2\mathfrak d_M(p^*,q^*)
+\eta|K(\widehat{p})-K(\widehat{q})|
+2\mathfrak d_M(\widehat{p},\widehat{q})\\
+2C_\eta \frac{B_{\gamma,M}(\log|\mathcal{N}_{m,\ve}^M|+\log(2/\delta))}{n}+\varepsilon_{\mathrm{opt},M}
\end{multline*}
for any \(\eta\in(0,1)\).
Since \(\eta<1\), and by the choice of \(q^*(p^*)\) and \(\widehat{q}(\widehat{p})\), it further holds that
\[
\eta|K(q^*)-K(p^*)|
+2\mathfrak d_M(p^*,q^*)
\leq L(p^*, q^*(p^*))
\leq \omega_m(\ve,M),
\]
and a similar bound holds for the quantity involving \(\widehat{q}(\widehat{p})\). Hence,
\[
(1-\eta)K(\widehat{p})
\leq (1+\eta)K(p^*)
+2\omega_m(\ve,M)
+2C_\eta \frac{B_{\gamma,M}(\log|\mathcal{N}_{m,\ve}^M|+\log(2/\delta))}{n}+\varepsilon_{\mathrm{opt},M}.
\]

\subsection{Proof of Lemma~\ref{lem:full-truncated-kl}}
\label{proof:lem:full-truncated-kl}
Since for all \(x\in A_M\) it holds that \(\pi_\xi(x)=(1-\tau_\xi(M))\pi_\xi^M(x)\), we have
\begin{eqnarray*}
    \KL(\pi_0\Vert \pi_\xi)
    &=&\int_{A_M} \pi_0(x)\log\frac{\pi_0(x)}{\pi_\xi(x)}\,dx
    +\int_{A_M^c} \pi_0(x)\log\frac{\pi_0(x)}{\pi_\xi(x)}\,dx \\
    &=&(1-\tau_0(M))\KL(\pi_0^M\Vert\pi_\xi^M)
    +(1-\tau_0(M))\log\frac{1-\tau_0(M)}{1-\tau_\xi(M)}
    +R_\xi(M).\label{KL_up3}
\end{eqnarray*}
Observing that \(1-\bar{\tau}_m(M)\leq 1-\tau_0(M), 1-\tau_\xi(M)\leq 1\) and hence \(\log(1-\tau_0(M)), \log(1-\tau_\xi(M))\in[\log(1-\bar{\tau}_m(M)),0]\), we further get that \(|(1-\tau_0(M))\log\frac{1-\tau_0(M)}{1-\tau_\xi(M)}|\leq d_m(M)\), and, since trivially, \(|R_\xi(M)|\leq \bar{R}_m(M)\), conclude~\eqref{eq:full-from-truncated} from~\eqref{eq:exact-full-truncated-kl}. Conversely,~\eqref{eq:exact-full-truncated-kl} implies 
\[
\KL(\pi_0^M\Vert\pi_\xi^M)
\leq \frac{1}{1-\tau_0(M)}\left(
\KL(\pi_0\Vert \pi_\xi)
+d_m(M) + \bar{R}_m(M)
\right).
\]
By~\eqref{eq:exp-envelope_gen},
\(
1-\tau_0(M)
\geq 1-\bar{\tau}_m(M)
\geq Z_\kappa^{-1}\int_{v_\star}^M e^{-K(s-v_\star)}\,ds,
\)
i.e., \(1-\tau_0(M)\geq c_M\) with
\[
c_M
:=\frac{\int_{v_\star}^M e^{-K(s-v_\star)}m_V(s)\,ds}{Z_\kappa}
=\frac{\int_{A_M} e^{-K(V(x)-v_\star)}\,dx}{\int_{\R^d} e^{-\kappa(V(x)-v_\star)}\,dx}
\in(0,1),
\]
the upper bound being due to finiteness of \(M\) alongside Assumption~\ref{ass:potential} guaranteeing continuity and coercivity of \(V\).

\subsection{Proof of Proposition~\ref{prop:full-mle-transfer}}
\label{proof:prop:full-mle-transfer}
The key idea is to extend the results obtained for the truncated experient to the non-truncated one, by evaluating the cost of resorting to truncation. To this end, define \(E_M:=\{X_1,\dots,X_n\in A_M\}\) and observe that \(X_1,\dots,X_n\) are i.i.d.\ with density \(\pi_0^M\) under \(\mathbb{P}_M:=\mathbb{P}(\cdot|E_M)\). Define \(L_n(\xi):=n^{-1}\sum_{i=1}^n\log\pi_\xi(X_i)\) and \(L_n^M(\xi):=\frac{1}{n}\sum_{i=1}^n \log\pi_\xi^M(X_i)\). By the same reasoning as in the proof of Lemma~\ref{lem:full-truncated-kl}, as well as the definition of \(\widehat{\Xi}_n\), it holds on \(E_M\), for all \(\xi\in\cX_m\), that
\[
L_n^M(\widehat{\Xi}_n)
-L_n^M(\xi)
=L_n(\widehat{\Xi}_n)-L_n(\xi)
-\log(1-\tau_{\widehat{\Xi}_n}(M))
+\log(1-\tau_\xi(M))
\geq -d_m(M),
\]
from which \(L_n^M(\widehat{\Xi}_n)\geq \sup_{\xi\in\cX_m} L_n^M(\xi)-d_m(M)\). Applying Proposition~\ref{prop:truncated-oracle} and Lemma~\ref{lem:full-truncated-kl}, we then get that
\begin{multline}
\KL(\pi_0^M\Vert\pi_{\widehat\Xi_n}^M)
 \leq
C_1\inf_{\xi\in\cX_m}\KL(\pi_0^M\Vert\pi_\xi^M)
 +C_2\frac{B_{\gamma_M}\{\mathcal H_m(\ve,M)+\log(2/\delta)\}}{n} \\
  +C_3\omega_m(\ve,M)
  +C_4d_m(M)+\bar{R}_m(M)
\end{multline}
with probability at least \(1-\delta\) under \(\mathbb{P}_M\), where \(C_i>0\), \(i\in\{1,\dots,4\}\), are some constants. Given~\eqref{eq:truncated-from-full}, the claim~\eqref{eq:full-mle-oracle-transfer} follows. Denoting the event that this inequality holds by \(G_M\), we get, since \(\mathbb{P}(E_M)=(1-\tau_0(M))^n\), 
that
\(
\mathbb{P}(G_M\cap E_M)
=\mathbb{P}(G_M|E_M)\mathbb{P}(E_M)
\geq (1-\delta)(1-\tau_0(M))^n
.
\)

\subsection{Proof of Proposition~\ref{prop:exact-forward-approx}}
\label{proof:prop:exact-forward-approx}
Denote \(e_m:=\xi_m-\Xi_0\) and observe that 
\[
G_m(r)
=\int_1^r \left(
\frac{1}{u\xi_m(u)}
-\frac{1}{u\Xi_0(u)}
\right)\,du
=-(\cK e_m)(r) 
+\int_1^r \frac{e_m^2(u)}{\xi_m(u)}w_0(u)\,du. 
\]
As was shown in Lemma~\ref{lem:dnn}, \(\Vert \cK e_m\Vert_{L^2(\nu_0)}\lesssim m^{-(\beta+1)}\). Define
\[
r_m(r):=\int_1^r
\frac{e_m^2(u)}{\xi_m(u)}w_0(u)\,du.
\]
Since \(\kappa\leq \xi_m,\Xi_0\), it holds that
\[
|r_m(r)|
\leq \frac{1}{\kappa^3}
\left|\int_1^r \frac{e_m^2(u)}{u}\,du\right|.
\]
Repeating the graded-cell argument used for \(\cK e_m\) gives
\[
\|r_m\|_\infty
\lesssim
\int_0^U\frac{|e_m(u)|^2}{u}\,\mathrm du
\lesssim
\Delta_{1,m}^{2\beta}
+
\sum_{j\geq2}
\frac{\Delta_{j,m}^{2\beta+1}}{r_{j-1,m}}
\lesssim
m^{-2\beta}
\lesssim m^{-(\beta+1)}
\]
for \(\beta>1\).

We next establish the uniform \(L^2(v_t)\)-bound. The cellwise
estimates in the proof of Lemma~\ref{lem:dnn} give
\(
\Vert\cK e_m\Vert_\infty
\lesssim m^{-\gamma}
\) with \(\gamma:=\min\{\vartheta\beta,\beta+1\}>1.
\)
Thus, setting \(\varepsilon_m:=\Vert G_m\Vert_\infty\), we have
\(
\varepsilon_m
\leq \Vert\cK e_m\Vert_\infty+\Vert r_m\Vert_\infty
\lesssim m^{-\gamma}
\lesssim m^{-1}.
\)
Recall that \(g_t=g_0+tG_m\), and write \(\widetilde{q}_t=g_t^{-1}\). Since
\(\Vert g_t-g_0\Vert_\infty\leq t\varepsilon_m\), we get 
\(
g_0(r)-t\ve_m 
\leq g_t(r)
\leq g_0(r)+t\ve_m
\)
for all \(r>0\), and, choosing 
\(r:=g_t^{-1}(y)\) with \(y\in\R\) fixed, using also that \(\widetilde{q}_0\) is increasing, get
\(
\widetilde{q}_0(y-t\varepsilon_m)
\leq \widetilde{q}_t(y)
\leq \widetilde{q}_0(y+t\varepsilon_m).
\)
Denoting 
\(F_t(\cdot)
:=\int_{\R^d}\widetilde{q}_t(\cdot-V(x))\,dx\)
and using the normalisation condition 
\(
F_t(\mu_t)=F_0(\mu_0)=1,
\)
we get by the above, since \(F_0\) is strictly increasing, that
\(
|\mu_t-\mu_0|\leq t\varepsilon_m.
\)
Consequently, for every \(x\in\R^d\),
\[
\left|g_0(\pi_t(x))-g_0(\pi_0(x))\right|
=\left|\mu_t-\mu_0-tG_m(\pi_t(x))\right|
\leq 2t\varepsilon_m.
\]
Since
\(
g_0'(r)=(r\Xi_0(r))^{-1}
\geq (Kr)^{-1},
\)
it follows that
\begin{equation}\label{eq:pit/pi0}
\left|\log\frac{\pi_t(x)}{\pi_0(x)}\right|
\leq K\left|
\int_{\pi_0(x)}^{\pi_t(x)}
\frac{1}{r\Xi_0(r)}\,dr
\right|
\leq K|g_0(\pi_t(x))
-g_0(\pi_0(x))|
\leq 2Kt\varepsilon_m.
\end{equation}
Moreover,
\[
G_m'(r)
=\frac{1}{r\xi_m(r)}-\frac{1}{r\Xi_0(r)}
=-\frac{e_m(r)}{r\xi_m(r)\Xi_0(r)},
\]
and hence
\(
\sup_{0<r\leq U}|rG_m'(r)|
\leq \Vert e_m\Vert_\infty \kappa^{-2}
\leq \kappa^{-2}C_{\mathrm{app}} m^{-\beta}.
\)
Together with~\eqref{eq:pit/pi0}, this yields
\[
|G_m(\pi_t(x))-G_m(\pi_0(x))|
\leq \sup_{0<r\leq U}
|rG'_m(r)|
\left|
\log\frac{\pi_t(x)}{\pi_0(x)}
\right|
\leq \frac{2K}{\kappa^2}C_{\mathrm{app}} t m^{-\beta}\varepsilon_m,
\]
and, since \(v_t\) is the distribution of \(\pi_t(X)\) under
\(X\sim\pi_t\), we conclude that,
for all \(0\leq t\leq 1\),
\begin{multline*}
\Vert G_m\Vert_{L^2(v_t)}
=\Vert G_m(\pi_t)\Vert_{L^2(\pi_t)}
\leq
\Vert G_m(\pi_0)\Vert_{L^2(\pi_t)}
\leq 2K\kappa^{-2}C_{\mathrm{app}} m^{-\beta}\varepsilon_m\\
\leq
e^{Kt\varepsilon_m}
\Vert G_m(\pi_0)\Vert_{L^2(\pi_0)}
+2K\kappa^{-2}C_{\mathrm{app}} m^{-\beta}\varepsilon_m\\
=
e^{Kt\varepsilon_m}
\Vert G_m\Vert_{L^2(v_0)}
+2K\kappa^{-2}C_{\mathrm{app}} m^{-\beta}\varepsilon_m.
\end{multline*}
Now,
\[
\Vert G_m\Vert_{L^2(v_0)}
\leq
\Vert\cK e_m\Vert_{L^2(v_0)}
+\Vert r_m\Vert_\infty
\leq C(C_{\mathrm{fwd}}) m^{-(\beta+1)},
\]
while \(m^{-\beta}\varepsilon_m\lesssim m^{-(\beta+1)}\).
Therefore,
\(
\sup_{0\leq t\leq1}
\Vert G_m\Vert_{L^2(v_t)}
\leq C_{\mathrm{path}} m^{-(\beta+1)}
\)
with \(C_{\mathrm{path}}\) depending on \(\kappa\), \(K\), \(C_{\mathrm{app}}\) and \(C_{\mathrm{fwd}}\).

Now observe that, by definition of \(\rho_t\), it holds \(\partial_t g_t(r)=\partial_t \int_1^r \rho_t(u)\,du=G_m(r)\), and in addition, \(g_t'(r)=\rho_t(r)\). Hence, we get in the same way as in~\eqref{eq:stat-time} that
\[
\dot \ell_t(x)
:=\partial_t \log\pi_t(x)
=\xi_t(\pi_t(x))\left(
\dot \mu_t - G_m(\pi_t(x))
\right),
\]
where
\(
\dot \mu_t 
=\left(\int_0^{a_t} \xi_t(r)G_m(r)v_t(dr)\right)\left(\int_0^{a_t} \xi_t(r)v_t(dr)\right)^{-1}
\)
due to the normalisation condition on \(\pi_t\) implying \(\int_{\R^d} \dot\ell_t(x)\pi_t(x)\,dx=0\). Since by the Cauchy-Schwarz inequality \(|\dot\mu_t|\leq (K/\kappa)\Vert G_m\Vert_{L^2(v_t)}\), it follows that
\begin{multline*}
    \Vert \dot\ell_t\Vert_{L^2(\Pi_t)}
    =\left(\int_{\R^d} \xi_t^2(\pi_t(x))
    |\dot\mu_t-G_m(\pi_t(x))|^2
    \pi_t(x)\,dx\right)^{1/2}\\
    =\left(\int_0 ^{a_t} 
    \xi_t^2(r)
    |\dot\mu_t-G_m(r)|^2
    v_t(dr)\right)^{1/2}\\
    \leq K\Vert \dot\mu_t-G_m\Vert_{L^2(v_t)}
    \leq K(1+K/\kappa) 
    \Vert G_m\Vert_{L^2(v_t)}
    \lesssim 
    m^{-(\beta+1)}
\end{multline*}
up to a constant 
\(C(\kappa, K, C_{\mathrm{path}})\),
where the second equality is due to \(R_t=\pi_t(X_t)\) following \(v_t\).
From this, since it holds that 
\(
\Vert \log (\pi_{\xi_t}/\pi_0)\Vert_{\infty}
=\Vert \int_0^t \dot \ell_u\,du\Vert_\infty
\lesssim \Vert G_m\Vert_{\infty}=o(1)
\), we get
\begin{multline*}
\chi^2(\pi_{\xi_t}\Vert \pi_0)^{1/2}
=\Vert \pi_{\xi_t}/\pi_0-1\Vert_{L^2(\Pi_0)}
\leq \int_0^1 \Vert (\pi_{\xi_u}/\pi_0)
\dot\ell_u \Vert_{L^2(\Pi_0)}\,du\\
\leq e^{K\ve_m}
\int_0^1 \Vert \dot\ell_u\Vert_{L^2(\Pi_u)}
\,du
\leq C(\kappa, K, C_{\mathrm{path}}) 
e^{K\ve_m}
m^{-(\beta+1)}.
\end{multline*}
Since \(\ve_m\leq C m^{-\gamma}\), choosing \(m_0\geq 1\), one can bound 
\(
e^{K\ve_m}
\leq e^{KC}
\)
by some absolute constant. 
Putting \(r:=\pi_{\xi_m}/\pi_0\), one can then observe that
\begin{multline*}
    \KL(\pi_0\Vert \pi_{\xi_m})
    \leq \chi^2(\pi_0\Vert \pi_{\xi_m})
    =\int_{\R^d} \left(
    \frac{1}{r(x)}-1
    \right)^2\pi_{\xi_m}(x)\,dx\\
    =\int_{\R^d} 
    \frac{(r(x)-1)^2}{r(x)}\pi_0(x)\,dx
    \leq e^{\Vert \log r\Vert_\infty}
    \chi^2(\pi_{\xi_m}\Vert\pi_0)\\
    \leq C^2(\kappa, K, C_{\mathrm{path}})  e^{\Vert \log r\Vert_\infty}
    m^{-2(\beta+1)}.
\end{multline*}
Enlarging the constants if necessary, we conclude the claim.

\subsection{Proof of Theorem~\ref{thm:xi-upper}}
\label{proof:thm:xi-upper}
Choosing \(C(I,a_0,d,V,K)<\ve_{I,K,V,a_0}\) with 
\(\ve_{I,K,V,a_0}\) given in Lemma~\ref{lem:mode-localization}, one can apply~\ref{lem:exact-kl-modulus} to get an upper bound in terms of the KL-divergence. By Theorem~\ref{thm:oracle}, the bound~\eqref{eq:oracle} holds with probability at least \((1-\delta)(1-\tau_0(M_n))^n\). Choosing \(m\asymp\left(n/(b_n\log n)\right)^{1/(2\beta+3)}\) as in Theorem~\ref{thm:oracle} yields the claim.

\subsection{Proof of Lemma~\ref{lem:primitive-kl}}
\label{proof:lem:primitive-kl} 
By the same argument as in the proof of Proposition~\ref{prop:exact-forward-approx}, it holds that \(\dot q_t(s)=(\dot\mu_t-G(q_t(s)))/\rho_t(q_t(s))\) with 
\[
|\dot \mu_t|
=\left|
\frac{
\int_0^{a_{\xi_t}}
G(r)m_V(\mu_t-g_t(r))\,dr
}{
\int_{v_\star}^\infty
\xi_t(q_t(s))q_t(s)m_V(s)\,ds
}
\right|
\leq 
\frac{C_{V,J}}{\kappa}\Vert G\Vert_{L^1(J)},
\]
where \(C_{V,J}\) is the upper bound for \(m_V\) on \(S_J\) given by Remark~\ref{rem:coarea-bounds},
and \(|\dot\mu_t|\leq C_{V,J}|J|\kappa^{-1}\Vert G\Vert_\infty\). 
The path score is
\[
  \dot\ell_t(x):=\partial_t\log\pi_t(x)
  =\xi_t(q_t(\V(x))) \left(
  \dot\mu_t-G(q_t(\V(x)))
  \right),
\]
which can be bounded as
\begin{multline*}
    \sup_{0\leq t\leq 1}\Vert \dot\ell_t\Vert_{L^2(P_{\xi_t})}^2
    \leq 2K^2|\dot\mu_t|^2
    +2\int_{v_\star}^\infty 
    \xi_t^2(q_t(s))G^2(q_t(s))m_V(s)
    q_t(s)\,ds\\
    \leq 2\left(
    \frac{C_{V,J} K}{\kappa}
    \right)^2
    \Vert G\Vert_{L^1(J)}^2
    +2K\int_J
    G^2(r)m_V(\mu_t-g_t(r))\,dr
    \leq C_{1,J,\kappa,K,V}
    \Vert G\Vert_{L^2(J)}^2
\end{multline*}
with\newline
\(C_{1,J,\kappa,K,V}:=2KC_{V,J}\left(
    C_{V,J}K|J|\kappa^{-2}
    +1
    \right)\),
    and 
    \(
    \Vert\dot\ell\Vert_\infty \leq C_{2,J,\kappa,K,V}\Vert G\Vert_\infty
    \)
    with\newline
    \(
    C_{2,J,\kappa,K,V}
    :=K(C_{V,J}|J|\kappa^{-1}+1).
    \)
Since 
\[
|\log(\pi_t(x)/\pi_0(x))|\leq\int_0^t |\dot\ell_u(x)|\,du\leq C_{2,J,\kappa,K,V}\Vert G\Vert_\infty
\]
and \(\partial_t\pi_t=\pi_t\dot\ell_t\), this yields
\begin{multline*}
    \chi^2(\Pi_{\xi_1}\Vert \Pi_{\xi_0})^{1/2}
    =\left\Vert
    \frac{\pi_{\xi_1}-\pi_{\xi_0}}{\sqrt{\pi_{\xi_0}}}
    \right\Vert_{L^2(dx)}
    \leq \int_0^1 
    \left\Vert
    \frac{\pi_{\xi_t}\dot\ell_t}{\sqrt{\pi_{\xi_0}}}
    \right\Vert_{L^2(dx)}
    \,dt\\
    \leq \sup_{0\leq t\leq 1}
    \left\Vert
    \frac{\pi_t}{\pi_{\xi_0}}
    \right\Vert_\infty^{1/2}
    \Vert\dot\ell_t\Vert_{L^2(\Pi_{\xi_t})}
    \leq \sqrt{2C_{1,J,\kappa,K,V}}
    \Vert G\Vert_{L^2(J)},
\end{multline*}
provided that \(\Vert G\Vert_\infty\leq \log 2/C_{2,J,\kappa,K,V}\). Since 
\(\KL(\Pi_{\xi_1}\Vert \Pi_{\xi_0})\leq \chi^2(\Pi_{\xi_1}\Vert \Pi_{\xi_0})\), we prove \eqref{eq:primitive-kl-bound}.

\subsection{Proof of Theorem~\ref{thm:minimax-lower}}
\label{proof:thm:minimax-lower}
Divide the interval $I$ into $M\asymp b^{-1}$ disjoint intervals $I_j = [r_j, r_j+b]$. If the exact primitive anchor $1$ lies within $I$, omit the single sub-interval containing it. 
Because $\int_0^1 \psi(v)\,\dd v = 0$, omitting this cell guarantees that the integral of the local perturbation vanishes exactly on both boundaries of every remaining cell $j$; the number of active parameters remains $M \asymp b^{-1}$. 
Choose a smooth function \(\psi\) supported on \((0,1)\), having derivatives up to order \(\beta\) and satisfying \(\int_0^1 \psi(u)\,du=0\). 
Define \(\psi_j(r):=b^\beta \psi((r-r_j)/b)\) and, for \(\theta\in\{0,1\}^M\), set \(\rho_\theta(r) := \rho_\star(r) + \alpha \sum_{j=1}^M \theta_j \psi_j(r)\) and \(\xi_\theta(r) = (r\rho_\theta(r))^{-1}\). Let \(\alpha>0\) be small enough, so that \(\xi_\theta\in\mathfrak X_\beta(J)\). 
For \(\theta\) and \(\theta^{(j)}\) neighbouring strings differing only in the \(j\)-th coordinate it holds that 
\(
\rho_{\theta^{(j)}}-\rho_\theta
=\alpha(\theta_j^{(j)}-\theta_j)\psi_j
=\pm\alpha\psi_j
\),
from which it follows, using the definition of \(\rho_\cdot\),
\[
|\xi_{\theta^{(j)}}(r)-\xi_\theta(r)|
=\left|
\frac{\alpha\psi_j(r)}{r\rho_\theta(r)\rho_{\theta^{(j)}}(r)}
\right|
=\alpha r\xi_\theta(r)\xi_{\theta^{(j)}}(r)|\psi_j(r)|.
\]
Given that on \(I\) it holds \(r\in[r_-,r_+]\Subset(0,a_\star)\) and \(\xi_\cdot\in[\kappa,K]\), and since \(\Vert \psi_j\Vert_{L^2(I)}^2=b^{2\beta+1}\Vert\psi\Vert_{L^2(0,1)}^2\), this yields
\[
(\alpha r_- \kappa^2 \Vert\psi\Vert_{L^2(0,1)})^2
b^{2\beta+1}
\leq \Vert \xi_{\theta^{(j)}}-\xi_\theta\Vert_{L^2(I)}^2
\leq (\alpha r_+ K^2 \Vert\psi\Vert_{L^2(0,1)})^2
b^{2\beta+1}.
\]
Now define \(G_j(r):=\int_1^r (\rho_{\theta^{(j)}}(u)
-\rho_\theta(u))\,du\). By above, \(G_j\) vanishes outside 
\(J\), and satisfies the bounds 
\(\Vert G_j\Vert_\infty \leq \alpha b^{\beta+1}\Vert\Psi\Vert_\infty\) 
and 
\(\Vert G_j\Vert_{L^2(I)}^2\leq \alpha^2 b^{2\beta+3}\Vert\Psi\Vert_{L^2(0,1)}^2\) 
with \(\Psi(v):=\int_0^v \psi(u)\,du\). 
Choosing \(\alpha\) and \(b\) such that \(\Vert G_j\Vert_\infty\leq \log 2/(K(1+C_{V,J}|J|\kappa^{-1}))\), one can apply Lemma~\ref{lem:primitive-kl} to get  
\[
\KL(\Pi_{\xi_{\theta^{(j)}}}\Vert \Pi_{\xi_\theta})\leq C_{J,\kappa,K,V}\Vert G_j\Vert_{L^2(I)}^2\lesssim b^{2\beta+3} \quad \text{and} \quad  
\KL(\Pi_{\xi_{\theta^{(j)}}}^{\otimes n}\Vert \Pi_{\xi_\theta}^{\otimes n})\lesssim nb^{2\beta+3}.
\]
Choosing now \(b=c_0 n^{-1/(2\beta+3)}\) with \(c_0\) small enough, one obtains 
\(\KL(\Pi_{\xi_{\theta^{(j)}}}^{\otimes n}\Vert \Pi_{\xi_\theta}^{\otimes n})\lesssim c_0^{2\beta+3}\). Applying the standard testing form of Assouad's lemma (see, e.g.,~\cite{tsybakov2009}), which bounds the minimax risk via the Hamming cube parameter dimension $M$ multiplied by the squared separation distance, we then get 
\(\inf_{\wh\Xi_n}
  \sup_{\xi\in\mathfrak X_\beta(J)}
  \E_\xi\Vert\wh\Xi_n-\xi\Vert_{L^2(I)}^2 \gtrsim Mb^{2\beta+1}\). Since \(M\asymp b^{-1}\asymp n^{1/(2\beta+3)}\), we conclude~\eqref{eq:lower-risk}. The fixed-probability lower bound~\eqref{eq:lower-prob} follows from the same Assouad argument: the uniform total-variation bound keeps the minimax Hamming testing risk bounded away from zero, and the separation of the Assouad cube then yields an \(L^2(I)\)-error of order \(n^{-\beta/(2\beta+3)}\) with probability bounded below by a positive constant.

\section*{Auxiliary results}

\begin{lemma}
\label{lem:ode-stability}
For any candidates $\xi,\zeta\in\cX_m$, the truncated invariant densities satisfy
\(
 \mathfrak d_M(\pi_\xi^M,\pi_\zeta^M) \le C_{\mathrm{st}}(1+M)\norm{\xi-\zeta}_{\infty,[0,U]}
\)
with some constant \(C_{\mathrm{st}}>0\) depending on \(V\), \(\kappa\) and \(K\).
\end{lemma}
\begin{proof}
For two candidates \(\xi,\zeta\in\cX_m\), let \(\xi_t(r) = (1-t)\zeta(r)+t\xi(r)\). By the same argument as in~\eqref{eq:stat-time}, it holds that
\[
\partial_t \log q_t(s)
=\xi_t(q_{\xi_t}(s))\left(
\dot \mu_{\xi_t} + (\cK_t h)(q_{\xi_t}(s))
\right)
\]
with \(q_{\xi_t}\) and \(\mu_{\xi_t}\) defined as before, and \(h:=\xi-\zeta\). For \((\cK_t h)(q_{\xi_t}(s))\) it holds, since \(\kappa\leq \xi_t(r)\leq K\) for all \(r\geq 0\), 
\[
\left| (\cK_t h) (r) \right|
\leq \frac{\Vert h\Vert_{\infty,\,[0,U]}}{\kappa^2}|\log r|,
\quad r\geq 0.
\]
Taking \(r=q_{\xi_t}(s)\) and using~\eqref{eq:exp-envelope_gen}, we conclude that 
\[|(\cK_t h)(q_{\xi_t}(s))|\leq C_{V,\kappa,K}\Vert h\Vert_{\infty,\,[0,U]}(1+s)\]
with some 
\(C_{V,\kappa,K}:=C(V,\kappa,K)>0\) 
for all \(s\geq 0\). Similarly, for \(\dot\mu_t\) one gets, by the same argument as in~\eqref{eq:c0-functional},
\[
\dot\mu_{\xi_t}
=-\frac{\int_{v_\star}^\infty 
q_{\xi_t}(s)\xi_t(q_{\xi_t}(s))(\cK_t h)(q_{\xi_t}(s))m_V(s)\,ds}{\int_{v_\star}^\infty q_{\xi_t}(s)\xi_t(q_{\xi_t}(s))m_V(s)\,ds}.
\]
The denominator of the expression above is separated from zero due to~\eqref{qmv_norm},
while for the numerator we have
\begin{multline*}
\int_{v_\star}^\infty 
q_{\xi_t}(s)\xi_t(q_t(s))(\cK_t h)(q_{\xi_t}(s))m_V(s)\,ds\\
\leq KC_{V,\kappa,K}\Vert h\Vert_{\infty,\,[0,U]}
\int_{v_\star}^\infty 
(1+s)q_{\xi_t}(s)m_V(s)\,ds,
\end{multline*}
the integral being finite again by~\eqref{eq:exp-envelope_gen}. Enlarging the constant \(C_{V,\kappa,K}\) if necessary, we then conclude \(|\dot\mu_t|\leq C_{V,\kappa,K}\Vert h\Vert_{\infty,\,[0,U]}\) and \(\partial_t \log q_{\xi_t}\leq C_{V,\kappa,K}\Vert h\Vert_{\infty,\,[0,U]}(1+s)\) for all \(s\geq 0\). It follows that
\[
|\log q_\zeta(s) - \log q_\xi (s)|
\leq \int_0^1 |\partial_t \log q_t(s)|\,dt
\leq C_{V,\kappa,K}\Vert h\Vert_{\infty,\,[0,U]}(1+s)
\]
for some generic constant \(C_{V,\kappa,K}\) depending on \(V\), \(\kappa\), \(K\). Taking \(s=V(x)\), one obtains \(|\log \pi_\zeta(x)-\log\pi_\xi(x)|\leq C_{V,\kappa,K}
(1+M)\Vert h\Vert_{\infty,\,[0,U]}\). 

Observe now that for the truncated densities it holds, for all \(x\in A_M\), that
\[
\log\pi_\xi^M(x)-\log\pi_\zeta^M(x)
=\log \pi_\xi(x)-\log\pi_\zeta(x)
-\log(1-\tau_\xi(M))
+\log(1-\tau_\zeta(M)),
\]
where, by above, the first difference is bounded by \(C_{V,\kappa,K}(1+M)\Vert h\Vert_{\infty,\,[0,U]}\). To bound the remaining terms, one can note that, similar to before,
\begin{eqnarray*}
\left|
\frac{d}{dt}\left(1-\tau_{\xi_t}(M)\right)
\right|
&=&\left|
\int_{v_\star} ^M 
(\partial_t q_{\xi_t}(s))m_V(s)\,ds
\right|\\
&=&\left|
\int_{v_\star} ^M 
q_{\xi_t}(s)
(\partial_t \log q_{\xi_t}(s))m_V(s)\,ds
\right|\\
&\leq& C_{V,\kappa,K}(1+M)\Vert h\Vert_{\infty,\,[0,U]}
\int_{v_\star}^M q_{\xi_t}(s)m_V(s)\,ds\\
&=&C_{V,\kappa,K}(1+M)\Vert h\Vert_{\infty,\,[0,U]} 
(1-\tau_{\xi_t}(M)),
\end{eqnarray*}
from which 
\(|\frac{d}{dt}
\log(1-\tau_{\xi_t}(M))|
\leq C_{V,\kappa,K}(1+M)\Vert h\Vert_{\infty,\,[0,U]}\) and consequently 
\(|\log(1-\tau_\xi(M))
-\log(1-\tau_\zeta(M))|
\leq C_{V,\kappa,K}(1+M)\Vert h\Vert_{\infty,\,[0,U]}\). 
Taking, e.g., 
\(C_{\mathrm{st}}=2C_{V,\kappa,K}\) yields the claim.
\end{proof}

\begin{lemma}\label{lem:explicit-tails}
    Under Assumptions~\ref{ass:potential} and~\ref{ass:coefficient}, it holds that
    \begin{eqnarray*}
        d_m(M)
        \lesssim 
        \bar{\tau}_m(M)
        &\leq& C_\tau
        e^{-c_aM},
        \quad \forall\,c_a\in(0,\kappa),\\
        \bar{R}_m(M)
        &\leq& C_{R,\ve} e^{-(\Xi_0(0)-\ve)M},
        \quad \forall\,\ve\in(0,\Xi_0(0)),\\
        \Lambda_{\gamma_M}+B_{\gamma_M}
        &\leq& C_\gamma (1+M)
    \end{eqnarray*}
    for \(M\) large enough,
    where 
    \(
    C_\tau=C(V,\kappa,K,c_a),
    \)
    \(
    C_{R,\ve}=C(a_0, V, \kappa, K, L, \ve)
    \)
    and
    \(
    C_\gamma=C(V,\kappa,K)
    \)
    are some positive constants.
\end{lemma}
\begin{proof}
By~\eqref{eq:exp-envelope_gen}, 
\[
\left|
\log\frac{\pi_0(x)}{\pi_\xi(x)}
\right|
\leq \log\frac{a_+}{a_-}
+(K-\kappa)(V(x)-v_\star)
\leq \log\frac{a_+}{a_-}
+(K-\kappa)(M-v_\star),
\quad \forall\,x\in A_M,
\]
where \(a_-:=Z_\kappa^{-1}\). Since also
\(1-\tau_0(M)\geq 1-\bar{\tau}_m(M)\geq c_M\),
it holds 
\[
\frac{\pi_0^M(x)}{\pi_\xi^M(x)}
=\frac{\pi_0(x)}{\pi_\xi(x)}
\frac{1-\tau_\xi(M)}{1-\tau_0(M)}
\leq \frac{a_+}{c_M a_-}
e^{(K-\kappa)(M-v_\star)},
\]
and, choosing \(\gamma_M\geq c_Ma_-/a_+ e^{-(K-\kappa)(M-v_\star)}\), we get 
\[\Lambda_{\gamma_M}=\log(1/\gamma_M)
\leq C(V,\kappa,K) (1+M).\] 
For \(B_{\gamma_M}\), observing that \(f(u)=(u+2)(u-1+e^{-u})-u^2\) is equal to zero at \(u=0\) and possesses a non-negative derivative, one obtains 
\(B_{\gamma_M}\leq \Lambda_{\gamma_M}+2\),
yielding the bound for \(\Lambda_{\gamma_M}+B_{\gamma_M}\) for some enlarged constant \(C_\gamma\). 
For the tail probability, since
\[
\bar{\tau}_m(M)
\leq a_+ \int_{\{V(x)>M\}}
e^{-\kappa(V(x)-v_\star)}\,dx
\leq C_a e^{-c_a(M-v_\star)},
\quad \forall\,c_a<\kappa,
\]
where \(C_a:=a_+\int_{\R^d} e^{-(\kappa-c_a)(V(x)-v_\star)}\,dx<\infty\) 
by Assumption~\ref{ass:potential}, it holds 
\(\bar{\tau}_m(M)\leq C_\tau e^{-c_a M}\) for any \(c_a<\kappa\) with some 
\(C_\tau=C(V,\kappa,K,c_a)\). Consequently, 
\(d_m(M)\lesssim \bar{\tau_m}(M)\leq C_\tau e^{-c_a M}\) for \(M\) large enough. Finally, observe that, for all \(y\geq 0\),
\begin{multline*}
q_0(v_\star+y)
=a_0\exp\left\{
-\int_0^y \Xi_0(q_0(v_\star+u))\,du
\right\}\\
=a_0\exp\left\{
-\Xi_0(0)y
-\int_0^y \left(
\Xi_0(q_0(v_\star+u))-\Xi_0(0)
\right)\,du
\right\}\\
\leq C(a_0,V,\kappa,K,L)e^{-\Xi_0(0)y}
\end{multline*}
and
\begin{multline*}
\log\frac{q_0(v_\star+y)}{q_\xi(v_\star+y)}
=\log\frac{a_0}{a_\xi}
-\int_0^y \left(
\Xi_0(q_0(v_\star+u))
-\xi(q_\xi(v_\star+u))
\right)\,du\\
\leq \log\frac{a_0}{a_\xi}
+(K-\kappa)y.
\end{multline*}
It follows that
\begin{multline*}
    \bar{R}_m(M)
    \leq C_0(a_0,V,\kappa,K,L)
    \int_{\{V(x)>M\}}
    \left(
        \log\frac{a_+}{a_-}
        +(K-\kappa)(V(x)-v_\star)
    \right)\\
    {}\times
    e^{-\Xi_0(0)(V(x)-v_\star)}\,dx
    \leq
    C_{R,\ve}
    e^{-(\Xi_0(0)-\ve)(M-v_\star)},
\end{multline*}
for every fixed \(0<\ve<\Xi_0(0)\), where
\[
C_{R,\ve}
:=
C_0(a_0,V,\kappa,K,L)
\int_{\mathbb R^d}
\left(
    \log\frac{a_+}{a_-}
    +(K-\kappa)(V(x)-v_\star)
\right)
e^{-\ve(V(x)-v_\star)}\,dx
<\infty
\]
by Assumption~\ref{ass:potential}. Equivalently,
\(\limsup_{M\to\infty} \log\bar{R}_m(M)/(M-v_\star)\leq -\Xi_0(0)\).
\end{proof}

\begin{lemma}
\label{lem:holder-besov-interpolation}
For any bounded interval $J$ and $\beta>0$ fixed, there exists a constant $C_{\beta,J}$ such that, for all $F\in L^2(J)\cap\mathcal H^{\beta+1}(J)$,
\begin{equation}\label{eq:holder-besov-interpolation}
 \norm{F'}_{L^2(J)} \leq C_{\beta,J} \norm{F}_{L^2(J)}^{\beta/(\beta+1)} \norm{F}_{\mathcal H^{\beta+1}(J)}^{1/(\beta+1)}.
\end{equation}
\end{lemma}
\begin{proof}
First map \(J\) to the unit interval \((0,1)\) and extend \(F\) by a bounded extension operator to a compactly supported function \(\widetilde{F}\) on \(\R\) such that \(\Vert \widetilde{F}\Vert_{L^2(\R)}\leq C_J\Vert F\Vert_{L^2(J)}\) and \(\Vert \widetilde{F}\Vert_{\mathcal H^{\beta+1}(\R)}\leq C_{\beta,J}\Vert F\Vert_{\mathcal H^{\beta+1}(J)}\) for some constant \(C>0\). Choose a smooth kernel \(\eta\) integrating to one and satisfying \(\int_{\R} y^j\eta(y)\,dy=0\) for \(j=1,\dots,\lceil\beta\rceil-1\); note that for \(\beta>2\), this kernel is signed. For \(\ve\in(0,1]\) define 
\(\eta_{\ve}(x):=\ve^{-1}\eta(x/\ve)\) and decompose \(\widetilde{F}'=(\widetilde{F}'-\eta_\ve\star\widetilde{F}')+\eta_\ve\star\widetilde{F}'\). Since \(\widetilde{F}'\in\mathcal{H}^\beta\), for all \(x\in\R\) it holds that
\begin{multline*}
    \widetilde{F}'(x)
    -\eta_{\ve}\star\widetilde{F}'(x)
    =-\int_{\R} \eta_{\ve}(y)\left(
    \sum_{j=1}^{\lceil\beta\rceil-1}
    \frac{(-y)^j}{j!}
    (\widetilde{F}')^{(j)}(x)
    +R_{\beta}(x,y)
    \right)\,dy\\
    =-\int_{\R} \eta_{\ve}(y)R_\beta(x,y)\,dy,
\end{multline*}
the last equality being due to the moment cancellation of \(\eta\), where \(R_\beta\) is the Taylor remainder satisfying \(|R_\beta(x,y)|\leq C|y|^\beta \Vert \widetilde{F}'\Vert_{\mathcal{H}^\beta(\R)}\). Hence, \(\Vert \widetilde{F}'-\eta_\ve\star\widetilde{F}'\Vert_\infty\leq C\ve^\beta \Vert \widetilde{F}'\Vert_{\mathcal{H}^\beta}(\R)\), and, since \(\widetilde{F}\) and \(\eta\) are both compactly supported, 
\(\Vert \widetilde{F}'-\eta_\ve\star \widetilde{F}'\Vert_{L^2(\R)}
\leq C\ve^\beta\Vert \widetilde{F}\Vert_{\mathcal{H}^{\beta+1}(\R)}
\leq CC_{\beta,J} 
\ve^\beta\Vert F\Vert_{\mathcal{H}^{\beta+1}(J)}
\). 

For the smoothed part, by Young's convolution inequality it follows that
\[
\Vert \eta_\ve\star\widetilde{F}'\Vert_{L^2(J)}
\leq\Vert \eta_\ve'\star\widetilde{F}\Vert_{L^2(\R)}
\leq \Vert\eta'_\ve\Vert_{L^1(\R)}
\Vert \widetilde{F}\Vert_{L^2(\R)}
\leq C_\eta C_J\ve^{-1}\Vert F\Vert_{L^2(J)}.
\]
Combining the two estimates, we get for some \(C_{\beta,J,\eta}>0\),
\[
\Vert F'\Vert_{L^2(J)}
\leq C_{\beta,J,\eta}
\left(
\ve^\beta \Vert F\Vert_{\mathcal{H}^{\beta+1}(J)}
+\ve^{-1}\Vert F\Vert_{L^2(J)}
\right).
\]
The choice 
\(\ve:=(\Vert F\Vert_{L^2(J)}/\Vert F\Vert_{\mathcal{H}^{\beta+1}(J)})^{1/(\beta+1)}
\)
for \(F\neq 0\) yields the result.
\end{proof}

\begin{lemma}
\label{lem:exact-kl-modulus}
Suppose Assumptions~\ref{ass:potential} and~\ref{ass:coefficient} hold, and assume \(\kappa\leq\xi\leq K\) and \(\Vert \xi\Vert_{\mathcal{H}^{\widetilde{\beta}}([0,U])}+\Vert\Xi_0\Vert_{\mathcal{H}^{\widetilde{\beta}}([0,U])}\leq \widetilde{H}\) for some \(\widetilde{\beta}>0\). 
Let \(I=[r_-,r_+]\Subset(0,a_0)\) and assume that the mode satisfies \(a_\xi \ge r_+ + \eta_I\) for some fixed \(\eta_I>0\). Finally, suppose that Assumption~\ref{ass:regular-interior} holds for \(\xi_\circ=\Xi_0\).
Then 
\begin{equation}\label{eq:exact-kl-modulus}
 \norm{\xi-\Xi_0}_{L^2(I)} \leq C \KL(\pi_0\Vert\pi_\xi)^{\widetilde{\beta}/(2(\widetilde{\beta}+1))}
\end{equation}
for some constant 
\(
C:=C(I,V,\kappa,K,\widetilde{\beta},\widetilde{H})>0
\). 
\end{lemma}
\begin{proof}
Due to the Gaussian envelope~\eqref{eq:ou-envelope}, and since the mode strictly satisfies \(a_\xi \ge r_+ + \eta_I\) for a fixed margin \(\eta_I>0\) there exists \(T_I\Subset(v_\star,\infty)\) such that \(s_\xi(I)\cup s_0(I)\subseteq T_I\). The uniform margin \(\eta_I\) securely detaches \(T_I\) from the minimum-energy geometric degeneracy. On \(T_I\), the profiles satisfy \(0<c_I\leq q_\xi(s)\leq C_I\), which implies that their derivatives are bounded uniformly away from zero and infinity due to
\begin{equation}
  c_I\kappa \leq -q_\xi'(s) = q_\xi(s)\xi(q_\xi(s)) \leq C_I K,
  \quad \forall\,s\in T_I.
\end{equation}
Since \(q_\xi(s_\xi(r))=r=q_0(s_0(r))\), we get by the mean value theorem
\begin{equation}
  |s_\xi(r)-s_0(r)| \leq \frac{1}{c_I\kappa} |q_\xi(s_\xi(r))-q_\xi(s_0(r))| = \frac{1}{c_I\kappa} |q_0(s_0(r))-q_\xi(s_0(r))|,
\end{equation}
from which
\begin{multline*}
\Vert s_\xi-s_0\Vert_{L^2(I)}^2
\leq \frac{1}{(c_I\kappa)^2}
\int_I 
|q_\xi(s_0(r))-q_0(s_0(r))|^2\,dr\\
=\frac{1}{(c_I\kappa)^2} 
\int_{s_0(r_+)}^{s_0(r_-)}
|q_\xi(s)-q_0(s)|^2
|q'_0(s)|\,ds
\leq \frac{C_I K}{(c_I\kappa)^2} 
\Vert q_\xi-q_0\Vert_{L^2(s_0(I))}^2.
\end{multline*}
Now observe that
\begin{equation}
  \KL(\pi_0\Vert\pi_\xi) = \int_{v_\star}^\infty m_\V(s) \left[ q_0(s)\log\frac{q_0(s)}{q_\xi(s)} - q_0(s) + q_\xi(s) \right] \dd s \geq \frac{c_{I,V}}{2a_+}\norm{q_\xi-q_0}_{L^2(S_I)}^2
\end{equation}
by the generic inequality \(u\log(u/v)-u+v\geq (u-v)^2/(2a_+)\) for all \(u,v\in (0,a_+]\). By proven above, it follows that \(\Vert s_\xi-s_0\Vert_{L^2(I)}^2\leq \widetilde{C}_{I,V} \KL(\pi_0\Vert\pi_\xi)\) for some \(\widetilde{C}_{I,V}>0\). Moreover,
\[
\Vert \xi-\Xi_0\Vert_{L^2(I)} 
=\left(\int_I 
\left|
r\xi(r)\Xi_0(r)(s_\xi'(r)-s_0'(r))
\right|^2\,dr
\right)^{1/2}
\leq r_+ K^2\Vert s_\xi'-s_0'\Vert_{L^2(I)},
\]
and by the assumptions made \(\Xi_0,\xi\in\mathcal{H}^{\widetilde{\beta}}([0,U],\widetilde{H})\) with some \(\widetilde{\beta}>0\), which means that \((s_\xi'-s_0')\in\mathcal{H}^{\widetilde{\beta}}(I)\) as \(r\geq r_->0\). Applying Lemma~\ref{lem:holder-besov-interpolation} of Supplementary Material to \(F:=s_\xi-s_0\), we then conclude that 
\[ 
\Vert F'\Vert_{L^2(I)}\leq \bar{C}_{I,V} \Vert F\Vert_{L^2(I)}^{\widetilde{\beta}/(\widetilde{\beta}+1)}\Vert F\Vert_{\mathcal{H}^{\widetilde{\beta}+1}(I)}^{1/(\widetilde{\beta}+1)}
\] for \(\bar{C}_{I,V}>0\) some constant, where the first norm to the right-hand side is bounded by \(\sqrt{\KL(\pi_0\Vert\pi_\xi)}\) and the second by some finite constant, giving the claim.
\end{proof}

The following lemma guarantees the existence of the mode margin \(\eta_I\) required by Lemma~\ref{eq:exact-kl-modulus} under the considered assumptions.
\begin{lemma}
\label{lem:mode-localization}
Under assumptions~\ref{ass:potential} and~\ref{ass:coefficient}, every candidate density satisfies
\begin{equation}\label{eq:mode-localization}
 |a_\xi-a_0| \leq 
 C_{d,K,V}
 \norm{\pi_\xi-\pi_0}_{L^1(\R^d)}^{2/(d+2)} 
 \leq 2^{1/(d+2)}C_{d,K,V}
 \KL(\pi_0\Vert\pi_\xi)^{1/(d+2)}
\end{equation}
with some constant 
\(C_{d,K,V}=C(d,K,V)>0\).
Thus, for every 
\(I=[r_-,r_+]\Subset(0,a_0)\), 
setting \(\eta_I=(a_0-r_+)/2\), there exists \(\varepsilon_{I,K,V,a_0}>0\) such that the inequality 
\(\KL(\pi_0\Vert\pi_\xi)
\leq\varepsilon_{I,K,V,a_0}\) 
implies \(a_\xi \ge r_+ + \eta_I\).
\end{lemma}
\begin{proof}
Observe that by~\eqref{eq:q-ode-general}, it holds that
\begin{multline*}
|\pi_\xi(x)-a_\xi|
=|q_\xi(V(x))-q_\xi(v_\star)|
=\left|
-\int_{v_\star}^{V(x)}
q'_\xi(s)\,ds
\right|\\
=\left|
\int_{v_\star}^{V(x)}
q_\xi(s)\xi(q_\xi(s))\,ds
\right|
\leq Ka_+(V(x)-v_\star).
\end{multline*}
Hence, \(a_\xi\geq a_0\) implies
\(
\pi_\xi(x)-\pi_0(x)
\geq a_\xi-Ka_+(V(x)-v_\star)-a_0,
\)
and reverting the roles of \(a_\xi\) and \(a_0\), we arrive at \(|\pi_\xi(x)-\pi_0(x)|\geq \Delta_\xi/2\) with \(\Delta_\xi:=|a_\xi-a_0|\), provided that \(V(x)-v_\star\leq \Delta_\xi/(2Ka_+)\). Integrating the inequality, we then get 
\begin{equation}\label{eq:Lebh}
    \Vert \pi_\xi-\pi_0\Vert_{L^1(\R^d)}
    \geq \frac{\Delta_\xi}{2}
    \nu_V\left(
    \frac{\Delta_\xi}{2Ka_+}
    \right),
\end{equation}
where 
\[
\nu_V(h)
    :=\operatorname{Leb}(\{
    x\in\R^d\colon
    V(x)-v_\star\leq h
    \}),
    \, h\geq 0.
\]
Note that by Assumption~\ref{ass:potential}, since \(V\) is coercive, it attains its minimum, and since \(V\in C^2\), for a fixed radius \(\rho_\star>0\) it holds that
\[C_V:=\max\left( 
\sup_{y\in \overline{B(x_\star,\rho_\star)}}\Vert \nabla^2 V(y)\Vert_{\text{op}},\,
(K\rho_\star^2)^{-1}
\right)
<\infty,\] where \(x_\star\) is the minimiser of \(V\). By Taylor's formula one then has, for all \(x\in B(x_\star, \rho_\star)\),
\[
V(x)-v_\star
=\int_0^1 (1-t)(x-x_\star)^\top 
\nabla^2 V(x_\star+t(x-x_\star))
(x-x_\star)\,dt
\leq \frac{C_V}{2}\Vert x-x_\star\Vert_2^2,
\]
from which it follows that \(B(x_\star, \sqrt{2h/C_V})\subseteq \{x\in\R^d\colon V(x)-v_\star\leq h\}\), provided \(h\leq C_V \rho_\star^2/2\). In turn, this implies that 
\(
\nu_V(h)
    \geq \omega_d \left(
    2h/C_V
    \right)^{d/2}
    \gtrsim h^{d/2}.
\)
Taking \(h=\Delta_\xi/(2Ka_+)\) and noting that by definition of \(C_V\) it holds that 
\(
2h/C_V=\Delta_\xi/(Ka_+C_V)
\leq Ka_+\rho_\star^2/(Ka_+)
=\rho_\star^2,
\)
i.e., \(h\leq C_V\rho_\star^2/2\) is satisfied,  we hence get for~\eqref{eq:Lebh}
\(
\Vert \pi_\xi-\pi_0\Vert_{L^1(\R^d)}
\geq (\omega_d/2)(Ka_+C_V)^{-d/2}
\Delta_\xi^{1+d/2}.
\)
Pinsker's inequality yields the second inequality in~\eqref{eq:mode-localization}. Taking \(\ve_{I,K,V, a_0}:=0.5(\eta_I/C_{d,K,V})^{d+2}\), we get that the condition 
\(\KL(\pi_0\Vert\pi_\xi)
\leq \ve_{I,K,V, a_0}\) with 
\(\eta_I\) taken as in the formulation of the lemma implies that \(|a_\xi-a_0|\leq\eta_I\), and in particular, \(a_\xi\geq a_0-\eta_I=r_+\eta_I\). 
\end{proof}

\bibliographystyle{plainnat}
\bibliography{xi_mle_dnn_multivariate_holder_interpolation_corrected}
\end{document}